\documentclass[12pt]{article}
\usepackage[utf8]{inputenc}

\usepackage{mathtools}
\usepackage{amssymb,amsmath,amsthm}
\usepackage[dvipsnames]{xcolor}
\usepackage{algorithm}
\usepackage{dsfont}
\usepackage{algpseudocode}
\usepackage{mathtools}
\usepackage{fullpage}
\usepackage{natbib}
\usepackage{bm}

\newtheorem{theorem}{Theorem}
\newtheorem{theorem*}{Theorem}

\newtheorem{assumption}{Assumption}
\newtheorem{remark}{Remark}

\newtheorem{definition}{Definition}
\newtheorem{lemma}{Lemma}
\newtheorem{proposition}{Proposition}

\newcommand{\Ind}{\bm{1}} 

\newcommand{\Var}{\mathrm{Var}}

\newcommand{\Cov}{\mathrm{Cov}}

\newcommand{\argmin}{\mathop{\arg\min}}
\newcommand{\indep}{\perp \!\!\! \perp}

\mathtoolsset{showonlyrefs}

\title{Sequential causal inference in a single world of connected units}
\author{Aurelien Bibaut, Maya Petersen, \\
Nikos Vlassis, Maria Dimakopoulou, \\
Mark van der Laan}

\begin{document}

\maketitle

\begin{abstract}
    We consider adaptive designs for a trial involving $N$ individuals that we follow along $T$ time steps. We allow for the variables of one individual to depend on its past and on the past of other individuals. Our goal is to learn a mean outcome, averaged across the $N$ individuals, that we would observe, if we started from some given initial state, and we carried out a given sequence of counterfactual interventions for $\tau$ time steps. 
    
    We show how to identify a statistical parameter that equals this mean counterfactual outcome, and how to perform inference for this parameter, while adaptively learning an oracle design defined as a parameter of the true data generating distribution. Oracle designs of interest include the design that maximizes the efficiency for a statistical parameter of interest, or designs that mix the optimal treatment rule with a certain exploration distribution. We also show how to design adaptive stopping rules for sequential hypothesis testing.
    
    This setting presents unique technical challenges. Unlike in usual statistical settings where the data consists of several independent observations, here, due to network and temporal dependence, the data reduces to one single observation with dependent components. In particular, this precludes the use of sample splitting techniques. We therefore had to develop a new equicontinuity result and guarantees for estimators fitted on dependent data.
    Furthermore, since we want to design an adaptive stopping rule, we need guarantees over the joint distribution of the sequence of estimators. In particular, this requires our equicontinuity result to hold almost surely, and our convergence guarantees on nuisance estimators to hold uniformly in time. We introduce a nonparametric class of functions, which we argue is a realistic statistical model for nuisance parameters, and is such that we can check the required equicontinuity condition and show uniform-in-time convergence guarantees.
    
    We were motivated to work on this problem by the following two questions. (1) In the context of a sequential adaptive trial with $K$ treatment arms, how to design a procedure to identify in as few rounds as possible the treatment arm with best final outcome? (2) In the context of sequential randomized disease testing at the scale of a city, how to estimate and infer the value of an optimal testing and isolation strategy?
 \end{abstract}

\section{Introduction}


We consider the setting in which, given a set of $N$ individuals, a decision maker (or experimenter) alternatively, over a sequence of time points $t=1,2,\ldots$, assigns to each individual $i$ a treatment $A(t,i)$ and then collects a vector of measurements $L(t,i)$ on this individual. We consider individual and time point specific outcomes $Y(t,i)$ that can be defined from  $L(t,i)$. We also suppose that the decision maker can adapt the treatment assignment rule in response to past observations. 

In these situations, it is often natural to define the performance of a treatment rule in terms of the expectation average outcomes of the form $N^{-1} \sum_{i=1}^N Y(\tau, i)$ at some time point $\tau$, or as a function of such averages at different time points $\tau_1,\tau_2,\ldots$. Natural objectives the decision makers may want to pursue include learning as fast as possible the optimal treatment rule (a so-called \textit{pure exploration} goal), or to ensure that over a certain time period, the individuals experience outcomes as high as possible (a so-called \textit{regret minimization} objective).

This setting can arise in particular in business and public health applications. We consider two motivating examples.

\paragraph{First motivating example.} Suppose an infectious disease is circulating in a country, and that public health officials can access, for each inhabitant $i$, at each time point $t$, a vector of measurements $L(t,i)$ including the infection status, which we define as the outcome $Y(t,i)$, demographic characteristics and the set of people $i$ has been in contact with in the recent past. Suppose the government can assign to each individual $i$, at every time step $t$, a treatment $A(t,i)$ consistent of a certain set of restrictions on their daily activities. A question of interest is, given two candidates treatment rules, how to learn as quickly as possible which one is the most efficient.

\paragraph{Second motivating example.} Suppose that the administrators of a web platform wonder which of two versions of the user interface users like best in the long run. They select a set of $N$ users among all of the users of the platform, and assign from the beginning half of them to version $1$ ($A(t,i) = 1$ for every $i =1,\ldots,N/2$ and every $t \geq 1$) and the other half to version 2 ($A(t,i) = 2$ for every $i=N/2+1,\ldots,N$ and every $t \geq 1$). For each user $i$, at each time $t$, they collect of vector of measurements $L(t,i)$ on the user, which contains in particular measures of engagement with the platform, from which they define an outcome $Y(t,i)$. Given an arbitrary time point $\tau$, a question of interest is: how to find out as quickly as possible which user interface would have maximized the expected average outcome $N^{-1} \sum_{i=1}^N Y(\tau, i)$ at $\tau$?

\medskip

While traditionally the causal inference and sequential decision problems literature have focused on single time point interventions or multiple time points interventions on independents units, the data collect in many real world situations involve network dependence between individuals. In the infectious disease setting presented above, the network dependence arises from contagion effects. In the web platform example, adjusting the treatment rule of individuals at time $t$ as a function of the observed history of every individual up to time $t-1$ induces dependence between the trajectories of distinct individuals. In other business applications, there can be similar network effect arising due to word of mouth between socially connected users of the same service. Another source of association between units can arise from spillover effects, that is the effect of the treatment assignment of one individual on other individuals.

In this work, we define causal effects defined under temporal and network dependence, and we propose a methodology to design and analyze adaptive trials aiming at learning these causal effects. We propose a method to construct adaptive stopping rules for sequential hypothesis testing. We work under the key modelling assumption that the conditional distribution of the measurement vectors $L(t,i)$ given the past is constant across $i$ and $t$. This assumption is comparable to an homogeneity assumption in a Markov Decision Process setting. We will see further down that as a key consequence of this homogeneity assumption, the error rates of our estimators over this model are a function of $T \times N$, which acts as the effective sample size, even though under temporal and network dependence, we only observe a single independent draw of the data-generating distribution.

\subsection{Existing work}

The setting we study conjugates three topics often treated separately: causal inference under temporal and network dependence, and adaptive experimentation.

\paragraph{Causal inference with temporal dependence.} Temporal dependence in causal inference arises in longitudinal studies, and in particular in the dynamic treatment regime (DTR) literature for health applications, where patients are monitored across multiple time points, and a final outcome is measured at the end. In the DTR literature, dependence on the past is arbitrary and not identical across time points (as opposed to the homogeneous Markov Decision Process setting we discuss next), and convergence guarantees are stated in terms of the number of individuals $N$ enrolled in the study. In another strand of causal inference for longitudinal settings, authors assume that the trajectory of each individual can be modelled as an homogeneous Markov Decision Process (MDP) (see \cite{kallus2019efficiently} for the MDP model, and \cite{laan_chambaz_lendle2018} for a class of statistical models which include the MDP model), and convergence guarantees are stated in terms of the number of time points $T$ if one follows only one single individual, or in terms of $T \times N$, if one follows the trajectories of $N$ individuals over $T$ time steps. These works can be categorized in the Off-Policy Evaluation (OPE) sub-field of Reinforcement Learning (RL), where it is standard to model the trajectory of the system by an MDP.

\paragraph{Causal inference in networks.} Causal inference in networks has been studied by numerous authors (see e.g. \cite{hudgens_halloran2008, tchetgen_tchetgen-VanderWeele_2012, vdL2013, basse_airoldi2018, basse_feller_toulis2019, ogburn_diaz_sofrygin_vdL2020}). 

In network causal models, potential outcomes of an individual do not only depend on its own treatment history, but can also depend on the treatment history of other individuals it is connected to.

Applications include the study of infectious diseases and vaccines, spillover effects of advertising campaigns on social network platforms, and spillover effects in public policy interventions (see e.g. the aforementioned \cite{basse_feller_toulis2019}).

\paragraph{Adaptive experimentation.} Adaptive experimentation, a sub-field of sequential decision problems has been a very active field of study for more than 75 years, with seminal contributions dating back to the work of Wald on sequential probability ratio tests \citep{wald1945}, and the seminal multi-armed bandit paper of Robbins \citep{robbins1952}. The development of bandit algorithms was initially motivated by clinical trials with the goal of making these faster, and minimizing the opportunity cost of patients subjected to suboptimal treatments. Objective pursued in adaptive experimentation / sequential decision problems include (1) minimizing the cumulative regret, that is, over a fixed number of rounds, maximizing the sum of rewards collected / outcomes observed, and (2) inferential goals, such as identifying the best treatment arm with a certain predetermined level of confidence in as few rounds as possible (top arm identification in the fixed confidence setting), or to identify the best arm with as high a confidence level as possible, under a fixed number of rounds (best arm identification in the fixed confidence setting). In the bandit literature, it is usually assumed that rewards (and contexts in the contextual bandit setting) are independent of the past (this covers both the stochastic i.i.d. setting and the oblivious adversarial setting). In that sense, usual bandit methods aren't appropriate to deal with the case of trajectories of individuals with temporal dependence, which is the setting which interests us.

Bandit problems are a special case of reinforcement learning problems with trajectories of length $T=1$. In the general case however, reinforcement learning is concerned with trajectories of the system over multiple time points, and states and outcomes at one time point can depend on the state, outcomes, and treatment at previous time points. Note that under the MDP model, which is the standard model considered in RL, dependency on the past is fully captured by the latest state and treatment. The reinforcement literature literature is concerned with learning optimal policies, either in a sequential fashion, or in an off-policy fashion, and with evaluating policies from off-policy data.

\paragraph{Sequential adaptive experimentation in the statistics literature.} 
We see mainly two directions in which an experiment can be made adaptive to the past, and which the statistics literature has considered. The first one is the stopping rule. Contributions on adaptive stopping rules date back to the aforementioned work of Wald on sequential probability ratio tests \citep{wald1945}. The other component that can be chosen adaptively is the design, that is the stochastic rule that the experimenter uses to assign treatment in response to past observations and current covariates. An optimal design is defined with respect to a given statistical parameter and is the design that maximizes efficiency for that parameter, that is that leads to the smallest asymptotic variance of estimators of that parameter, and the highest power of tests of hypothesis defined from that parameter. \cite{vdL2008} proposed a comprehensive methodology for sequential adaptive trials in the single and multiple time point settings, for independent individuals (as opposed to the network interference setting)

\subsection{Contributions and comparison with past work}



Our theoretical contributions are the following. We present our causal model, and we define our causal parameter as a mean outcome under a post-intervention distribution under this causal model. Under identifiability conditions, this post-intervention distribution equals a G-computation formula. We thus define formally our statistical parameter as the corresponding mean outcome under the G-computation formula.

We derive its efficient influence function (EIF), and thereby the semiparametric efficiency bound. Our statistical model subsumes in particular the homogeneous MDP model with independent trajectories. To the best of our knowledge, this work is the first to provide a formal derivation of the EIF of mean outcomes under a $G$-computation formula under the homogeneous MDP model.

We provide a Targeted Maximum Likelihood Estimator (TMLE) and a one-step estimator of our target parameter for generic sequential adaptive designs.  We show under certain conditions that a certain process obtained by rescaling in time the sequence of estimates converges weakly to a Wiener process. In particular, this gives us the asymptotic distribution of our TMLE and one-step estimator. The conditions include in particular conditions on the design.
We show that, for designs that converge to a fixed limit design, our estimator has asymptotic variance equal to the variance of the canonical gradient of the target parameter, which we conjecture equals the semiparametric efficiency bound. We use the results on the time-rescaled sequence of estimates to design a method to construct adaptive stopping rules for sequential hypothesis testing. 

As mentioned earlier, some technical challenges arise from the fact that observations of different individuals at different time points are a priori dependent, and from the fact that we need to characterize the joint distribution of the sequence of estimates so as to be able to design an adaptive stopping rule. In particular the dependence between individuals implies that we cannot use the usual sample splitting techniques, which in the construction of semiparametric estimators allow to circumvent Donsker conditions \citep{klaassen1987, Zheng2011, kallus_uehara2020-DRL}. We therefore had to derive an almost sure equicontinuity result for empirical processes generated by weakly dependent data. The almost sure convergence aspect is key to obtaining guarantees on the joint distribution of the sequence of estimators. For the latter purpose, we also needed uniform-in-time convergence guarantees for nuisance estimators. We derived an exponential deviation bound for empirical risk minimizers fitted on weakly dependent data, which allows us to control these uniformly in time. Both the equicontinuity results and the results on empirical risk minimizers stem from a maximal inequality for empirical processes generated by weakly dependent data (that is from sequences that satisfy a certain mixing condition), which we obtain by applying an adaptive chaining device, a classical empirical process technique pioneered by \citep{ossiander1987}, to a Bernstein inequality for mixing sequences, proven by \cite{merlevede2009}.

\subsection{Paper organization} 

In section \ref{section:problem_formulation}, we describe our causal model, our causal parameter, the statistical model, the statistical target parameter parameter, and the class of adaptive designs we consider. In section \ref{section:structural_results}, we derive the efficient influence function (canonical gradient) of our target w.r.t. our statistical model, and we study the robustness properties of the EIF. In section \ref{section:estimators}, we provide a Targeted Maximum Likelihood Estimator and a one-step estimator of our parameter of interest, and we derive their convergence guarantees. In section \ref{section:Lipschitz_HAL_class_and_nuisance_estimation}, we introduce a class of functions that we use for nuisance modelling, and for modelling the EIF, and we give guarantees for empirical risk minimizers over it. In section \ref{section:adaptive_stopping_rules}, we show how to construct an adaptive stopping rule for sequential hypothesis testing. In section \ref{section:learning_optimal_design}, we discuss adaptive learning of the optimal design.

\section{Problem formulation}\label{section:problem_formulation}

\subsection{Observed data}

An experimenter interacts with an environment consisting of $N$ individuals indexed by $i=1,\ldots,N$, over rounds $t=1,\ldots,T$. At each round $t$, the experimenter first assigns the treatment vector $A(t) := (A(t,1), \ldots, A(t,N))$ to the $N$ individuals, where $A(t,i)$ is the treatment assigned to individual $i$ at time $t$, and then observes for each individual $i$, a vector $L(t,i)$ of time-varying covariates and outcomes. Let $L(t) := (L(t,1), \ldots, L(t,N))$. We denote $O(t,i) := (A(t,i), L(t,i))$ the data observed for individual $i$ at time $t$, and $O(t) = (A(t), L(t))$ the data observed for the all of the $N$ individuals at round $t$, $\bar{O}(t) := (O(1), \ldots, O(t))$ the data available at round $t$, $L^-(t) := \bar{O}(t-1)$, the data observed before $L(t)$, $A^-(t) := (\bar{O}(t-1), L(t))$, the data observed before $A(t)$, $L^-(t,i) = (L^-(t), L(t,1), \ldots, L(t,i-1))$, and $A^-(t,i) = (A^-(t), A(t,1), \ldots, A(t,i-1))$.

Under this notation, the data observed throughout the course of the trial is thus $\bar{O}(T)$, which we will also denote $O^{T,N}$ to make explicit the dependence on the number of individuals $N$.

\subsection{Causal model}

\paragraph{Formal definition of the causal model.} 

We suppose that there exists a set of deterministic functions $\{f_{A(t,i)}, f_{L(t,i)} : t \in [T], i \in [N] \}$, and a set of unobserved random variables $U := (U^A, U^L)$, with $U^L := (U^L(t,i) : t \in [T], i \in [N])$ and $U^A := (U^A(t,i) : t \in [T], i \in [N] )$, such that, for every $(t,i) \in [T] \times [N]$,
\begin{align}
    A(t,i) =& f_{A(t,i)}(\bar{L}(t-1), \bar{A}(t-1),  U^A(t,i)),\label{eq:NPSEM_eq1} \\ 
    L(t,i) =& f_{L(t,i)}(\bar{A}(t), \bar{L}(t-1), U^L(t,i)). \label{eq:NPSEM_eq2}
\end{align}

We place no restriction at this point on the functional form of the functions $f_{A(t,i)}$, $f_{L(t,i)}$. 
The set of equations \eqref{eq:NPSEM_eq1}-\eqref{eq:NPSEM_eq2} form a so-called \textit{Nonparametric Structural Equation Model} (NPSEM) (see e.g \cite{pearl_2009}). Let
\begin{align}
    \mathcal{M}^{T,N}_F := \left\lbrace P^{T,N}_F : P_U, f_{A(t,i)}, f_{L(t,i)}, \widetilde{c}_{A(t,i)}, \widetilde{c}_{L(t,i)} : t  \in [T], i \in [N] \right\rbrace,
\end{align}
be the set of probability distributions $P^{T,N}_F$ over the domains of $(O^{T,N}, U)$ induced by the NPSEM as the distribution $P_U$ of the unmeasured variable vector $U$, and the functions  $f_{A(t,i)}, f_{L(t,i)}, \widetilde{c}_{A(t,i)}, \widetilde{c}_{L(t,i)}$ vary freely. The set $\mathcal{M}^{T,N}_F$ is our so-called \textit{causal model}.

We denote $P_{0, F}^{T,N}$ the true distribution of $(O^{T,N},U)$. In the remainder of the article, we will use the subscript ``0'' to indicate true probability distributions or components thereof. 
Note that the full data distribution fully determines the observed data distribution. 

\paragraph{Counterfactual data and post-intervention distribution.} We now describe a counterfactual scenario in which the connectivity of the nodes and the intervention assigned to them is not as under the NPSEM above.

Let $\{g^*_{s,j} : s \in [\tau], j \in [N]\}$ be a collection of stochastic interventions at nodes $\{ A(s,j), s \in [T], j \in [N] \}$, that is, for every $(s,j)$, $g^*_{s,j}$ is a distribution over treatment arms conditional on $\bar{A}(t-1), \bar{L}(t-1)$.

Let $O^{*,T,N} := (O^*(t,i) : t \in [T], i \in [N])$, with $O^*(s,j) := (A^*(s,g), L^*(s,j))$ be the counterfactual data set generated from $U$ by the following NPSEM, obtained from the NPSEM \eqref{eq:NPSEM_eq1}-\eqref{eq:NPSEM_eq2} by replacing the intervention nodes by the counterfactual interventions $g^*_{s,j}$:
\begin{align}
    A^*(s,j) \sim & g_{s,j}^*(\cdot \mid \bar{L}(s-1), \bar{A}(s-1)), \\ 
    L^*(s,j) =& f_{L(s,j)}(\bar{A}(s-1), \bar{L}(s-1), U^L(s,j)).
\end{align}\
The distribution of $(O^{*,T,N},U)$ is the so-called \textit{post-intervention distribution} of the full data. We use the notation $P^{*,T,N}_F$ for the post-intervention distribution of the full data.

\paragraph{Causal target parameter.} Let $\tau \geq 1$ be an arbitrary time point. We define as our causal target parameter as a certain mean outcome at time point $\tau$, under the post-intervention distribution:
\begin{align}
    \Psi^F_\tau(P^{T,N}_F) = E_{P^{*,T,N}_F} \left[ Y^*(\tau) \right],
\end{align}
where $Y^*(\tau) := N^{-1} \sum_{i=1}^N Y^*(\tau, i)$, where $Y^*(\tau,i)$ is a unit-specific outcome at time $\tau$, defined as $Y^*(\tau, i) := f_Y(L^*(\tau, i))$ for a certain function $f_Y$. In words, $\Psi^F_\tau(P^{T,N}_F)$ is the mean outcome at time $\tau$, in the counterfactual scenario where the treatment mechanism is $g^*$.

\paragraph{Identifiability.} We use the notation $P^{T,N}$ for a generic distribution over the domain of the observed data $O^{T,N}$ and we denote $P^{T,N}_0$ the true distribution of the observed data.
As mentioned above, any distribution $P^{T,N}_F$ on the domain of the full data fully determines a corresponding distribution $P^{T,N}$ on the domain of the observed data. For any $P^{T,N}_F$, we say that the causal parameter $\Psi^F_\tau(P^{T,N}_F)$ is identifiable if we can write it as a function of the corresponding observed data distribution $P^{T,N}$.

The parameter $\Psi^F_\tau(P^{T,N}_F)$ is identifiable under the following two assumptions.

\begin{assumption}[Sequential randomization]\label{assumption:sequential_randomization}
For any $(t,i)$, $A(t,i) {\perp\!\!\perp} L^*(t,i) \mid A(t,i)^-$.
\end{assumption}

\begin{assumption}[Positivity]\label{assumption:positivity} For any $(t,i)$, $a \in \mathcal{A}$, and $l(t,i)^-$ such that $P_0\left[ L(t,i)^- = l(t,i)^-\right] > 0$,
\begin{align}
    P_0 \left[ A(t,i) = a \mid L(t,i)^- = l(t,i)^-\right] > 0.
\end{align}
\end{assumption}
It can be shown under assumptions \ref{assumption:sequential_randomization} and \ref{assumption:positivity} that the post-intervention distribution of $O^*$ under $P^{*,T,N}_F$ equals the following G-computation formula:
\begin{align}
    P^{T,N}_{g^*}(O^{T,N}) := \prod_{s=1}^\tau \prod_{j=1}^N P(L(s,j) \mid L(s,j)^-) g^*_{s,j}(A(s,j) \mid A(s,j)^-).
\end{align}
Note that the factors of $P^{T,N}_{g^*}$ are the conditional distributions $P(L(t,i) \mid L(t,i)^-)$, which are known if we know $P^{T,N}$, and the known counterfactual intervention $g^*$. Therefore, under assumptions \ref{assumption:sequential_randomization} and \ref{assumption:positivity}, there exists a mapping $\Psi_\tau$ such that 
\begin{align}
    \Psi^F_\tau(P^{T,N}_F) = \Psi_\tau(P^{T,N}).
\end{align}
The quantity $\Psi_\tau(P^{T,N}_0)$ is then our statistical target parameter. For $\gamma \in (0,1)$, which we interpret as a \textit{discount factor}, we also define the following other target parameter, derived from $\Psi_\tau$:
\begin{align}
    \Psi_{\tau, \gamma}(P^{T,N}) := \sum_{\tau' \geq \tau} \gamma^{\tau'-\tau} \Psi_{\tau'}(P^{T,N}).
\end{align}
This parameter is the typical target (in particular in the case $\tau=1$) in the Off-Policy Evaluation problem in reinforcement learning (see e.g. \cite{kallus_uehara2020-DRL, kallus2019efficiently}). As the analysis of $\Psi_{\tau,\gamma}$ follows from the analysis of $\Psi_{\tau}$, in the rest of the paper, we treat $\Psi_{\tau}$ as our main object of interest, and we will simply denote it $\Psi$ when there is no ambiguity about $\tau$.

\subsection{Statistical model}

The statistical model is the set of distributions we believe to contain the true data-generating distribution $P_0^{T,N}$. We denote it $\mathcal{M}^{T,N}$. We suppose that all of the elements of $\mathcal{M}^{T,N}$ admit a density w.r.t. a common dominating measure $\mu$. For any $P^{T,N}$, we denote $p^{T,N} := dP^{T,N} / d \mu$.
 
From the chain rule, any such $p^{T,N}$ can be factorized as a product of conditional densities as follows:
\begin{align}
    p^{T,N}(o^{T,N}) := \prod_{t=1}^T \prod_{i=1}^N g_{t,i}(a(t,i) \mid a(t,i)^-) q_{t,i}(l(t,i) \mid l(t,i)^-).
\end{align}
The conditional densities $q_{t,i}$ are a fact of nature and represent how each $L(t,i)$ responds to past interventions $A(s,j)$, and depends on the vectors $L(s,j)$ of past time-varying covariates and outcomes of for each individual $j \in [N]$, for every time point $s < t$. We refer to it as the \text{uncontrolled} part of the data-generating process of the trial, as the experimenter does not have control over it. The factors $g_{t,i}$, in our randomized experimental setting, are in control of the experimenter, and represent the set of stochastic decision rules she follows to assign treatment at each time step.

The above decomposition places no restriction on the temporal and network dependence in the observed data. We make an assumption on the complexity of the dependence allowed by supposing that each $L(t,i)$ can depend on the past of the trial only through a summary measure of the history of a fixed number of individuals. 

\begin{assumption}[Conditional independence given summary measure]\label{assumption:cond_indep_given_summary_measure}
    There exists a Euclidean set $\mathcal{C}_L \subset \mathbb{R}^{d_1}$, for some $d_1 \geq 1$, and a set of deterministic functions $c_{L(t,i)}$, $t \in [T]$, $i \in [N]$, with image included in $\mathcal{C}$, such that, for every $t \in [T]$, $i \in [N]$,
    \begin{align}
        q_{t,i}(l(t,i) \mid l(t,i)^-) = q_{t,i}(l(t,i) \mid c_{L(t,i)}( l(t,i)^-) ).
    \end{align}
    We denote $C_L(t,i) := c_{L(t,i)}(L(t,i)^-)$. 
\end{assumption}
The vector $C_L(t,i)$, which lies in the Euclidean set $\mathcal{C}_L$, plays the role of a finite dimensional summary measure of the past, which is such that $L(t,i)$ is independent of its past when conditioning on this summary measure. Following terminology used in existing works \citep{boruvka2017, laan_chambaz_lendle2018}, we will refer to it as the ``context'' preceding the node $L(t,i)$.

Without any further assumptions, it is a priori not possible to obtain consistent estimators from a single draw of $O^{T,N}$. For consistent estimation to be possible, we need the likelihood to exhibit a repeated factor. We therefore make the following assumption.

\begin{assumption}[Homogeneity]\label{assumption:homogeneity}
The factors $q_{t,i}$ are constant across values of $i$ and $t$, that is there exists a common conditional density $q$ such that $q_{t,i} = q$ for every $t$ and $i$.
\end{assumption}

We can now state a formal definition of our statistical model.
\begin{definition}[Statistical model]
Fix $\mu$ a summary measure on the domain of $O^{T,N}$. We define our statistical model $\mathcal{M}^{T,N}$ as the set of distributions $P^{T,N}$ over the domain of $O^{T,N}$ that satisfy assumptions \ref{assumption:cond_indep_given_summary_measure} and \ref{assumption:homogeneity}.
\end{definition}

\begin{remark}
We emphasize that, as we are in the setting of a controlled trial, the factors $g_{t,i}$ are known.
\end{remark}

\begin{remark}
Observe that any distribution $P^{T,N}$ in our statistical model $\mathcal{M}^{T,N}$ is fully determined by $q$, and $\mathcal{M}^{T,N}$ is indexed by the set of conditional densities $\mathcal{M}_q$ that we believe to contain $q_0$. Here, we assume that $\mathcal{M}_q$ is a saturated nonparametric model, that is for any $c \in \mathcal{C}_L$, the tangent space of $\{ l \mapsto q(\cdot \mid c) : q \in \mathcal{M}_q \}$ is equal to the Hilbert space 
\begin{align}
    L^2_{0,c}(q) :=\left\lbrace (l,c) \to f(l,c) : \int f^2(l,c)q(l \mid c) dl < \infty \text{ and } \int f(l,c) q(l \mid c) dl = 0 \right\rbrace.
\end{align}
\end{remark}

\begin{remark}
Under the homogeneity assumption, the target parameter $\Psi(P^{T,N})$ depends on $P^{T,N}$ only through the common conditional density $q = q(P^{T,N})$. Therefore, there exists a mapping $\Psi^{(1)}$ such that $\Psi(P^{T,N}) = \Psi^{(1)}(q(P^{T,N}))$.
\end{remark}

In constructing and analyzing our estimators, we will require the following assumption.
\begin{assumption}\label{assumption:C_L_sufficient}
For any $s \in [\tau]$, $j, k \in [N]$, 
\begin{align}
    E_{q,g^*} \left[ Y(k) \mid L(s,j), C_L(s,j) \right] = E_{q,g^*} \left[ Y(k) \mid L(s,j), L(s,j)^-\right],
\end{align}
where $E_{q,g^*}$ is the expectation operator under the G-computation formula $P^{\tau,N}_{g^*}$.
\end{assumption}
Making assumption \ref{assumption:C_L_sufficient} on top of the previous two assumptions \ref{assumption:cond_indep_given_summary_measure} and \ref{assumption:homogeneity} defines a new statistical model $\mathcal{M}^{T,N}_1$, which is a subset of our previously defined statistical model $\mathcal{M}^{T,N}$. Note that this statistical model a priori depends on $g^*$. We want to emphasize the following: while we will derive in the next section the canonical gradient $D(P^{T,N})$ of our target parameter $\Psi$ w.r.t. the larger model $\mathcal{M}_{T,N}$, we will use the assumptions defining the smaller model $\mathcal{M}_1^{T,N}$ to derive a more tractable representation $D_1(P^{T,N})$ of this canonical gradient $D(P^{T,N})$, which we will use to build our estimators. As a result, estimators that achieve asymptotic variance equal to the variance of $D_1(P^{T,N})$ can only be locally efficient w.r.t. the model $\mathcal{M}^{T,N}$: they can achieve the efficiency bound for $\mathcal{M}^{T,N}$ at $P^{T,N}$ only if $P^{T,N} \in \mathcal{M}_1^{T,N}$.

Finally, in some special cases that we discuss next it might be realistic to make the following set of three assumptions.
\begin{assumption}[Context decomposition]\label{assumption:C_L_decomposition}
For any $t \in [T]$, $i \in [N]$, the context summary mapping $c_{L(t,i)}$ can be decomposed as
\begin{align}
    c_{L(t,i)}(l(t,i)^-) = (a(t,i), c_A^{g,g^*}(t,i)),
\end{align}
where $c_A^{g,g^*}(t,i)$ is a context for the node $A(t,i)$ of the form 
\begin{align}
    c_A^{g,g^*}(t,i) = c_{A(t,i)}^{g,g^*}(a(t,i)^-),
\end{align}
where $c_{A(t,i)}^{g,g^*}$ is a known deterministic function with image in a Euclidean set $\mathcal{C}_A \subset \mathbb{R}^{d_2}$, such that, for any $a(t,i)$, $a(t,i)^-$,
\begin{align}
    g_{t,i}(a(t,i) \mid a(t,i)^-) =& g_{t,i}(a(t,i) \mid c_A^{g,g^*}(t,i)) \\
    \text{and } g^*_{t,i}(a(t,i) \mid a(t,i)^-) =& g^*_{t,i}(a(t,i) \mid c_A^{g,g^*}(t,i)).
\end{align}
\end{assumption}

\begin{assumption}[Individual outcomes independent from other trajectories]\label{assumption:indiv_outcomes_indep_other_trajs}
For any $q \in \mathcal{M}_q$, it holds under the corresponding G-computation formula $P_{g^*}^{\tau,N}$ that $Y(k) \indep O(s,j)$ for any $s \in [\tau]$ and $k \neq j$.
\end{assumption}

\begin{assumption}[Observed treatment homogeneity]\label{assumption:trt_homogeneity} The treatment mechanisms $g_{t,i}$ are constant across $t$ and $i$, that is, there exists a conditional density $g$ such that $g_{t,i} = g$ for every $t$ and $i$.
\end{assumption}


Making assumptions \ref{assumption:C_L_decomposition}, \ref{assumption:indiv_outcomes_indep_other_trajs} and \ref{assumption:trt_homogeneity} on top of assumptions \ref{assumption:cond_indep_given_summary_measure},  \ref{assumption:homogeneity} and \ref{assumption:C_L_sufficient} defines a new statistical model $\mathcal{M}^{T,N}_2$ such that $\mathcal{M}^{T,N}_2 \subset \mathcal{M}^{T,N}_1 \subset \mathcal{M}^{T,N}$. Here too, we emphasize that we will use these additional assumptions to obtain a simplified representation $D_2(P^{T,N})$ of the canonical gradient of $\Psi$ w.r.t. the larger model $\mathcal{M}^{T,N}$, but that we won't derive the canonical gradient w.r.t. the smaller model $\mathcal{M}^{T,N}_2$. As a result, estimators achieving asymptotic variance equal to the variance of $D_2(P^{T,N})$ can only be efficient w.r.t. $\mathcal{M}^{T,N}$ if $P^{T,N} \in \mathcal{M}^{T,N}_2$.

\subsection{Network structures covered by the statistical models considered in this article}

\subsubsection{Network structures covered by the larger model $\mathcal{M}^{T,N}$}

Note that assumption \ref{assumption:homogeneity} does not restrict the network structure. The network structures covered by model $\mathcal{M}^{T,N}$ are therefore those that satisfy assumption \ref{assumption:cond_indep_given_summary_measure}.

\paragraph{Example 1: finite memory, bounded number of contacts.} Consider the setting where $L(t,i)$ is allowed to depend on $L(t,i)$ only through a summary measure of the history over last $t_0$ steps of a set $F_L(t,i)$ of at most $N_0$ friends. Then, if we allow the dimension of $\mathcal{C}_L$ to be as large as $(2 t_0 + 1) N_0$, assumption \ref{assumption:cond_indep_given_summary_measure} holds for the summary measure
\begin{align}
    c_{L(t,i)}(l(s,i)^-) := & \left( (a(s,j) : s=t-t_0,\ldots,t,\ j \in F_L(t,i)), \right. \\
    & \left. \ (l(s,j) : s=t-t_0,\ldots,t-1,\ j \in F_L(t,i)) \right).
\end{align}

\paragraph{Example 2: finite memory, dependence on aggregate measures only.} Consider the set of distributions $P^{T,N}$ such that $L(t,i)$ depends on a finite set of aggregate measures of the trial's history observed before the nodes $L(t)$. Consider for example, in the infectious disease example, and suppose that the intervention $A(t,i)$ is whether the $i$ wears a mask at $t$. Such aggregate measures could include summaries of $\bar{L}(t-1)$ such as the average infection rate across the entire population at time steps $t-t_0,\ldots,t-1$, the average infection rate among individuals $i$ has been in contact with at time steps $t_0,\ldots,t-1$. Aggregate summary measures of $\bar{A}(t)$ could include the fraction of people wearing masks in the population at $t=t-t_0,\ldots,t$, and the number of individuals at time steps $t=t-t_0,\ldots,t$ not wearing masks that $i$ has been in contact with. Note that in this setting, we can build a summary function mapping histories into a fixed dimensional set $\mathcal{C}_L$ without imposing restrictions on the number of contacts of each individuals.

\subsubsection{Networks structures covered by the model $\mathcal{M}_1^{T,N}$}

\paragraph{Example 3: disjoint independent clusters modelled by MDPs.} Suppose that $H_1, \allowbreak \ldots, \allowbreak H_n$ form a partition of $[N]$, and that there exists a constant $N_0$ such that $|H_k| \leq N_0$ for every $k$. We say that each $H_k$ is a \textit{cluster}. For any cluster $H$, denote $A(t,H) := (A(t,i) : i \in H)$, $L(t,H) := (L(t,i) : i \in H)$, $O(t,H) := O(t,i) : i \in H)$. For any $i$, let $k(i)$ be the cluster $i$ belongs to. Suppose that
\begin{align}
    q(L(t,i) \mid L(t,i)^-) = q(L(t,i) \mid A(t,H_{k(i)}), L(t-1,H_{k(i)})),
\end{align}
that is, $L(t,i)$ depends on the nodes preceding it only through the latest treatment vector $A(t,H_{k(i)})$ of the individuals in the same cluster, and on the latest measurement vector $L(t,H_{k(i)})$ of individuals in the cluster. Suppose that
\begin{align}
    g^*_{t,i}(a(t,i) \mid a(t,i)^-) = g^*_{t,i}(a(t,i) \mid l(t-1, H_{k(i)}) ),
\end{align}
that is under the counterfactual intervention, $A(t,i)$ depends only on the latest vector of measurement vector $L(t-1,H_{k(i)})$ of the individuals in the same cluster as $i$.
Let 
\begin{align}
    c_{L(t,i)}(l(t,i)^-) = &\left( (l(t,j) : j < i, j \in H_{k(i)}), a(t,H_{k(i)}), l(t-1,H_{k(i)}) \right).
\end{align}
Then it is immediate that assumption \ref{assumption:cond_indep_given_summary_measure} holds. Since in general $A(t, H_{k(i)})$, $L(t,H_{k(i)})$ do not block the dependence between the nodes $\{L(t,j) : j < i, j \in H_{k(i)} \}$ and $L(\tau,k)$, for $k \in H_{k(i)}$, $\tau > t$, we include these in the context summary measure to ensure that $L(\tau, k) \indep L(t,i)^- \mid C^L(t,i)$. It is then straightforward to check that assumption \ref{assumption:C_L_sufficient} holds. As a result, the class of network structures described in this example is covered by model $\mathcal{M}_1^{T,N}$. 

Note that the network structure under the observed treatment mechanism $g$ and the counterfactual treatment mechanism $g^*$ do not need to be the same. Note that the assumptions defining model $\mathcal{M}_1^{T,N}$ do not place any restriction on how $A(t,i)$ might depend on the past under $g$.

\paragraph{Example 4: treatment limits social interactions, $g^*$ forces individuals to stay in clusters.} Let $H_1,\ldots,H_n$ be disjoints clusters of at most $N_0$ individuals forming a partition of $[N]$. In our infectious disease example, we take these clusters to be households. We define the treatment as follows: $A(t,i) = 1$ if individual $i$ can meet with people outside of her household at time $t$, and $A(t,i) = 0$ if not. Regardless of treatment status, we suppose that $L(t,i)$ depends on the nodes $L(t,i)^-$ preceding it only through the history over the latest $t_0$ time steps of a set $F_L(t,i)$ of a most $N_1 \geq N_0$ individuals. We further suppose that $L(t,i)$ can only depend on the history over the last $t_0$ time steps of individuals $i$ is allowed to meet. Define the censoring indicator
\begin{align}
    \Delta(t,i,j) = \Ind\left( \left\lbrace a(t,i) = 1 \text{ and } j \in F_L(t,i) \right\rbrace \text{ or } \left\lbrace a(t,i) = 0 \text{ and } j \in H_{k(i)} \right\rbrace \right),
\end{align}
and the history summary mapping 
\begin{align}
    c_L{(t,i)}(l(t,i)^-) := & \left( ((l(t,j) \Delta(t,i,j), \Delta(t,i,j) ) : j < i, j \in F_L(t,i)), \right.\\
    & \ ( (a(s,j) \Delta(s,i,j), \Delta(s,i,j) ) : j \in F_L(t,i), s = t - t_0,\ldots, t),\\
    & \left. \ ( (l(s,j) \Delta(s,i,j), \Delta(s,i,j) ) : j \in F_L(t,i), s = t - t_0, \ldots, t-1) \right).
\end{align}
Under the intervention $g^*$ that deterministically assigns $A(s,j) = 0$ to every individual $j \in [N]$ at every time point $s \in [\tau]$, it is straightforward to check that the above defined finite dimensional summary measure mapping verifies assumptions \ref{assumption:cond_indep_given_summary_measure} and \ref{assumption:C_L_sufficient}.

\subsubsection{Network structures covered by the model $\mathcal{M}^{T,N}_2$}

\paragraph{Example 5: $N$ independent MDPs under $g^*$.} Consider our second motivating example in which the administrators of a web platform want to identify the treatment arm $a$ that has highest long term outcome. Formally, let $\tau \geq 1$ be a time point at which we deem the outcome to be a ``long-term'' outcome, and for each arm $a = 1,2$, let $g^{*,a}_{t,i} (a(t,i) \mid a(t,i)^-) := \Ind(a(t,i) = a)$, the intervention that always assigns deterministically arm $a$.  We define the long terms outcomes of each arm as $\Psi(a)(P^{T,N}) := E_{q,g^{*,a}} [Y(\tau)]$. Suppose that 
\begin{align}
    q(l(t,i) \mid l(t,i)^-) =& q(l(t,i) \mid a(t,i), l(t-1,i), \\
    \text{and } g^*_{t,i}(a(t,i) \mid a(t,i)^-) =& g^*_{t,i}(a(t,i) \mid l(t,i)), \\
\end{align}
that is, under $g^*$, individual trajectories are independent MDPs. Suppose further that the observed treatment mechanism satisfies
\begin{align}
    g_{t,i}(a(t,i) \mid a(t,i)^-) =& g_{t,i}(a(t,i) \mid l(t,i), \theta(t)),
\end{align}
where $\theta(t) \in \mathbb{R}^{d_3}$ is a summary measure of the entire trial's history $\bar{o}(t-1)$ which contains the parameters of the design. In this case, assumption \ref{assumption:C_L_decomposition} is satisfied for $c_A^{g,g^*}(t,i) := (a(t,i), \theta(t))$. 

It is straightforward to check that assumptions \ref{assumption:cond_indep_given_summary_measure}, \ref{assumption:C_L_sufficient} and \ref{assumption:indiv_outcomes_indep_other_trajs} then hold.

We now discuss what type of adaptive designs can satisfy the constraint expressed in the previous display.

If the goal of the experimenter is to minimize regret, appropriate adaptive designs might include some type of variant of UCB, or some type of $\varepsilon$-greedy design. In the UCB case, so as to parameterize the design at time $t$, it suffices for $\theta(t)$ to contain estimates $(\widehat{\Psi}_t(a) : a=1,2)$ of the long term outcomes under each arms, and of the standard deviations, which we denote $(\widehat{\sigma}_t(a): a=1,2)$ of these estimates. In the case of an $\varepsilon$-greedy design, $\theta(t)$ needs only to contain $(\widehat{\Psi}_t(a) : a=1,2)$. If the goal 

If the goal of the experimenter is to maximize the efficiency of an estimator of the contrast $\Psi(2)(P^{T,N}_0) - \Psi(1)(P^{T,N}_0)$, an appropriate design might some type of Neyman allocation design (see e.g. \cite{vdL2008}). Such a design can be defined based on estimates $(\widehat{\sigma}_t(a): a=1,2)$ of the standard deviations of the canonical gradients of $\Psi(1)$ and $\Psi(2)$.

\subsection{Comparison with the statistical model studied in past works}

\cite{laan_chambaz_lendle2018} and \cite{kallus2019efficiently} consider single individual trajectories or multiple independent single trajectories, that is they work in the case $N=1$.  The model studied in \cite{laan_chambaz_lendle2018} is $\mathcal{M}^{T,N}$ under $N=1$.  The homogeneous MDP model studied in \cite{kallus_uehara2020-DRL} is a special case of the model $\mathcal{M}^{T,N}_2$, in the case $N=1$.

We point out that neither of these two works provide a formal proof of the derivation of the canonical gradient of their target parameters w.r.t. the statistical models they consider. These can be obtained from the results of this article.

\cite{vdL2013} and \cite{ogburn_diaz_sofrygin_vdL2020} study a more general setting where $g$ is unknown and $q_{t,i}$ is not assumed to be constant across time points. This means that their statistical model contains $\mathcal{M}^{T,N}$. Note that the the canonical gradient of $\Psi$ w.r.t. their larger model is not equal to the canonical gradient of $\Psi$ w.r.t. $\mathcal{M}^{T,N}$. We point out nevertheless that the derivation of the canonical gradient of $\Psi$ w.r.t. $\mathcal{M}^{T,N}$ follows from a straightforward adaptation of the proof technique of \cite{vdL2013}.

\section{Structural properties of our target paremeter}\label{section:structural_results}

\subsection{Efficient influence function}

In the upcoming theorem, we present the canonical gradient $D$ of $\Psi$ w.r.t. $\mathcal{M}^{T,N}$. We also provide two simplified representations of $D$ when $P^{T,N}$ is in $\mathcal{M}^{T,N}_1$, and in $\mathcal{M}^{T,N}_2$, respectively. As pointed out in the previous section, we stress out that these are simplified representation of the canonical gradient w.r.t. $\mathcal{M}^{T,N}$ when $P^{T,N}$ belongs to submodels of $\mathcal{M}^{T,N}$, and not the expressions of the canonical gradient w.r.t. these submodels.

Let $h_{t,i}^L$ and $h_{t,i}^{*,L}$, be the marginal densities of $C^L(t,i)$ under $P^{T,N}$ and the corresponding G-computation formula $P_{g^*}^{\tau,N}$, and let $\bar{h}_{T,N}^L := (T N)^{-1} \sum_{t=1}^N \sum_{i=1}^N h_{t,i}^L$ Under assumption \ref{assumption:C_L_decomposition}, the contexts are $C_A^{g,g^*}(t,i)$ is defined. We then denote $h_{t,i}^A$ and $\bar{h}^*_{t,i}$ the marginal densities of $C_A^{g,g^*}(t,i)$ under $P^{T,N}$ and $P^{\tau,N}_{g^*}$, respectively, and we let $\bar{h}^A_{T,N} := (T N)^{-1} \sum_{t=1}^T \sum_{i=1}^N h_{t,i}^A$. Since we will refer more often to $h_{t,i}^{*,L}$ and $\bar{h}^L_{T,N}$, than to the other marginal densities, we will often simply denote them $h_{t,i}^*$ and $\bar{h}_{T,N}$.

\begin{theorem}[Representation of the canonical gradient]\label{thm:EIF_representation}
The canonical gradient of $\Psi$ w.r.t. $\mathcal{M}^{T,N}$ at $P^{T,N}$ is given by
\begin{align}
    D(q)(o^{T,N}) = \frac{1}{T N} \sum_{t=1}^T \sum_{i=1}^N \bar{D}_{T,N}(q)(c^L(t,i), l(t,i)),
\end{align}
where, for any $c^L$ and $l$,
\begin{align}
    \bar{D}_{T,N}(c^L, l) = \sum_{s=1}^\tau \sum_{j=1}^N \frac{h^L_{t,i}(c^L)}{\bar{h}^L_{T,N}(c^L)} & \left\lbrace E_{q,g} \left[Y g^* / g \mid L(t,i) = l, C^L(t,i) = c^L \right] \right. \\
    & \left. - E_{q,g} \left[Y g^* / g \mid C^L(t,i) = c^L \right] \right\rbrace. \label{eq:EIF_rep1}
\end{align}
If $P^{T,N} \in \mathcal{M}_1^{T,N}$, then we can represent $\bar{D}_{T,N}$ as follows:
\begin{align}
   \bar{D}_{T,N}(c^L, l) = \sum_{s=1}^\tau \sum_{j=1}^N \frac{h^{*,L}_{t,i}(c^L)}{\bar{h}^L_{T,N}(c^L)} & \left\lbrace E_{q,g^*} \left[Y \mid L(t,i) = l, C^L(t,i) = c^L \right] \right. \\
   & \left. - E_{q,g^*} \left[Y  \mid C^L(t,i) = c^L \right] \right\rbrace. \label{eq:EIF_rep2}
\end{align}
Furthermore, if $P^{T,N} \in \mathcal{M}^{T,N}_2$, then the following representation of $\bar{D}_{T,N}$ holds:
\begin{align}
    \bar{D}_{T,N}(q)(c^L, l) = \sum_{s=1}^\tau \sum_{j=1}^N \omega_{s,j}(c^A) \eta_{s,j}(a \mid c^A) &\left\lbrace  E_{q,g^*}[Y \mid L(t,i) = l, A(t,i) = a, C^A(t,i) = c^A] \right.\\
    & \left. - E_{q,g^*} [Y \mid A(t,i) = a, C^A(t,i) = c^A]  \right\rbrace, \label{eq:EIF_rep3}
\end{align}
where $\omega_{s,j} := h^{*,A}_{t,i} / \bar{h}^A_{T,N}$, and $\eta_{s,j} := g^*_{s,j} / g_{s,j} = g^*_{s,j} / g_{s,j}$ (since under assumption \ref{assumption:trt_homogeneity}, $g_{s,j} = g$ for some $g$ common across values of $s$ and $j$).
\end{theorem}

\subsection{First order expansion and robustness properties}

Let $P^{T,N} \in \mathcal{M}^{T,N}$. Denote $q = q(P^{T,N})$ and $q_0 = q(P^{T,N}_0)$. Let 
\begin{align}
    R(q,q_0) :=& \Psi(P^{T,N}) - \Psi(P^{T,N}_0) + E_{P_0^{T,N}} \left[D(q)(O^{T,N}) \right]\\
    =&\Psi^{(1)}(q) - \Psi^{(1)}(q_0) + E_{P_0^{T,N}} \left[D(q)(O^{T,N}) \right].
\end{align}
We like to view the equivalent representation 
\begin{align}
    \Psi(P^{T,N}) - \Psi(P^{T,N}_0) = -E_{P_0^{T,N}} \left[D(q)(O^{T,N})\right] + R(q,q_0)
\end{align}
as a functional first order Taylor expansion of the difference $\Psi(P^{T,N}) - \Psi(P^{T,N}_0)$, in which we view $R(q, q_0)$ as a remainder term, which we will show is second order. We say that a remainder term $R'(q,q_0)$ is second order if it can be written as a sum of terms such that every term has a factor of the form $\prod_k (\eta_k(q) - \eta_p(q))^{\alpha_k}$, with $\sum_k \alpha_k \geq 2$. In the usual sense, we say that a remainder term $R'(q,q_0)$ is robust (or equivalently we say that the canonical gradient from which it is formed is robust) if it can be rewritten as $R_1'((\eta_1(q),\ldots, \eta_p(q)), (\eta_1(q_0),\ldots,\eta_p(q_0))$, with $\eta_1(q), \ldots, \eta_p(q)$ variation independent nuisance parameters, and is equal to zero is $\eta_i(q) = \eta_i(q_0)$ for every $i$ in a subset $\mathcal{I} \subset [p]$, $\mathcal{I} \neq [p]$. Note that if a remainder term is second order w.r.t. variation independent parameters, then it is robust in the usual sense.

Unfortunately, the canonical gradient of $\Psi$ w.r.t. $\mathcal{M}^{T,N}$ is not robust in the usual sense, but we can show that it is second order and robust in a weaker sense, in which the nuisance $\eta_1(q),\ldots, \eta_p(q)$ are not variation independent.

We give two results that show that the remainder term $R$ is second order and robust in this weaker sense. These two results correspond to respectively representations \eqref{eq:EIF_rep2} and \eqref{eq:EIF_rep3} of the canonical gradient.

For pairs of indices $(t,i)$ and $(s,j)$, we write $(s,j) < (t,i)$ (resp. $(s,j) > (t,i)$) if $(s,j)$ comes strictly before (resp. strictly after) $(t,i)$ in the column ordering of indices. For any $q$, we denote
\begin{gather}
    q_{t,i}: o^{T,N} \mapsto q(l(t,i) \mid c^L(t,i)), \qquad q_{-(t,i)} : o^{T,N} \mapsto \prod_{(s,j) \neq (t,i)} q(l(s,j) \mid c^L(s,j)) \\
    q_{(t,i)-} : o^{T,N} \mapsto \prod_{(s,j) < (t,i)} q(l(s,j) \mid c^L(s,j)),\\
     \text{and }  q_{(t,i)+} : o^{T,N} \mapsto \prod_{(s,j) > (t,i)} q(l(s,j) \mid c^L(s,j)).
\end{gather}

\begin{theorem}\label{thm:first_second_order_result}
Suppose that $P \in \mathcal{M}^{T,M}_1$. Then, we can rewrite $R(q,q_0)$ as $R_1(\bar{h}_{T,N}, \bar{h}_{0,T,N}, q, q_0)$ where $R_1(\bar{h}_{T,N}, \bar{h}_{0,T,N}, q, q_0)$ satisfies
\begin{align}
    R_1(\bar{h}_{T,N}, \bar{h}_{0,T,N}, q, q_0) = &  R_1(\bar{h}_{T,N}, \bar{h}_{0,T,N}, q, q_0) - R_1(\bar{h}_{0,T,N}, \bar{h}_{0,T,N}, q, q_0) \\
    &+ R_1(\bar{h}_{0,T,N}, \bar{h}_{0,T,N}, q, q_0),
\end{align}
where,
\begin{align}
    &R_1(\bar{h}_{T,N}, \bar{h}_{0,T,N}, q, q_0) - R_1(\bar{h}_{0,T,N}, \bar{h}_{0,T,N}, q, q_0) \\
    &= \sum_{s=1}^\tau \sum_{j=1}^N \int \left(h_{s,j}^* \frac{\bar{h}_{0,T,N} - \bar{h}_{T,N}}{\bar{h}_{T,N}}\right)(c) (q_0 - q)(l \mid c) 
    \times  E_{q, g^*}[Y \mid L(t,i) = l, C(t,i) = c] dl dc,
\end{align}
and
\begin{align}
    R_1(\bar{h}_{0,T,N}, \bar{h}_{0,T,N}, q, q_0) = \sum_{s=1}^\tau \sum_{j=1}^N E_{q_{(s,j)-}, (q - q_0)_{(s,j)}, q_{0,(s,j)+} - q_{(s,j)+}, g^*} Y.
\end{align}
\end{theorem}

If $P^{T,N} \in \mathcal{M}^{T,N}_2$, we can further simplify the representation of the remainder term, as the following theorem shows.
\begin{theorem}\label{thm:second_second_order_result}
Suppose that $P \in \mathcal{M}^{T,N}_2$. Denote $\omega = (\omega_{s,j}: s \in [\tau], j \in [N])$. We can then rewrite $R(q,q_0)$ as $R_2(\omega, \omega_0, q, q_0)$ where the latter satisfies that
\begin{align}
    R_2(\omega, \omega_0, q, q_0) = \sum_{s=1}^\tau \frac{1}{N} \sum_{j=1}^N \int & \bar{h}^A_{0,T,N}(c^A) g^*_{s,j}(a \mid c^A) (\omega_{s,j} - \omega_{0,s,j})(c^A) (q - q_0)(l \mid a, c^A) \\
    & \times E_{q,g^*} \left[ Y(j) \mid L(s,j) = l, A(s,j) = a, C^A(s,j) = c^A\right] dl da dc^A.
\end{align}
From the expression above, $R_2(\omega, \omega_0, q, q_0)$ if $\omega = \omega_0$ or $q = q_0$.
\end{theorem}

\begin{remark}
In the above theorem, $\omega$ and $q$ are not variation independent components of $P^{T,N}$. In fact, since we know $g$, $\omega$ is fully determined by $q$.
\end{remark}

\begin{remark}
The proof of theorem \ref{thm:second_second_order_result} relies on the fact that we know the treatment mechanism $g$ since we are in a controlled trial, while the proof of theorem \ref{thm:first_second_order_result} does not.
\end{remark}

\section{Construction and analysis of our estimators}\label{section:estimators}

Let $\widehat{q}_{T,N}$, be an estimator of $q_0$.

\paragraph{TMLE estimator.} Let $\widehat{q}^*_{T,N}$ be an estimator of $q_0$ obtained from $\widehat{q}_{T,N}$ by the TMLE targeting step such that it solves approximately the EIF equation:
\begin{align}
    \frac{1}{T N} \sum_{t=1}^T \sum_{j=1}^N \bar{D}_{T,N}(\widehat{q}^*_{T,N})(L(t,i), C(t,i)) = o((TN)^{-1/2}).
\end{align}
We refer the reader to the Targeted Learning methodology papers and books \citep{TMLE_paper-vdL2006, TLbook1-vdL_Rose_2011, one-step-TMLE-vdL2016, TLbook2-vdL_Rose_2018} for details on the TMLE targeting steps.

We define our TMLE estimator as the following plug-in estimator:
\begin{align}
    \widehat{\Psi}^{\mathrm{TMLE}}_{T,N} := \Psi(\widehat{q}^*_{T,N})
\end{align}

\paragraph{One-step estimator.} The one-step estimator is defined as 
\begin{align}
    \widehat{\Psi}^{\mathrm{1-step}}_{T,N} := \Psi(\widehat{q}_{T,N}) + \frac{1}{T N} \sum_{t=1}^T \sum_{i=1}^N \bar{D}_{T,N}(\widehat{q}_{T,N})(L(t,i), C(t,i)). 
\end{align}

In what follows, we restrict our analysis to the TMLE estimator since the analysis for the 1-step estimator is identical. We will just denote $\widehat{\Psi}_{T,N} := \widehat{\Psi}^{\mathrm{TMLE}}_{T,N}$. The following theorem gives a decomposition of the difference between $\widehat{\Psi}_{T,N}$ and its target $\Psi(q_0)$. 

\begin{theorem}[TMLE expansion]\label{thm:TMLE_expansion} We have that 
\begin{align}
    \widehat{\Psi} - \Psi(q_0) :=& M_{1,T,N}(q_0) + M_{2,T,N}(\widehat{q}^*_{T,N}, q_0) + R(\widehat{q}^*_{T,N}, q_0),
\end{align}
with 
\begin{align}
    M_{1,T,N}(q_0) =& \frac{1}{T N} \sum_{t=1}^T \sum_{i=1}^N \bar{D}_{T,N}(q_0)(L(t,i), C(t,i)) \\
    M_{2,T,N}(q, q_0) =& \frac{1}{T N} \sum_{t=1}^T \sum_{i=1}^N (\delta_{L(t,i), C(t,i)} - P_{q_0, h_{0,t,i}}) \left(\bar{D}_{T,N}(q) - \bar{D}_{T,N}(q_0)\right),
\end{align}
and $R(q,q_0)$ is as defined in section \ref{section:structural_results} above.
\end{theorem}

The first term is the sum of a martingale difference sequence, and the process $$\{x \sqrt{T N }M_{1,xT,N}(q_0) : x \in [0,1] \}$$ can be shown, using a functional central limit theorem for martingale triangular arrays, to converge weakly, as $T \to \infty$ and under fixed $N$, to a Wiener process $\sigma_{0,\infty, N}^2 W$, with $W$ a standard Wiener process and $\sigma_{0,\infty, N}^2$ the limit of the variance under a certain limit distribution, of $\lim_{T \to \infty} \bar{D}_{T,N}(q_0)$ (we make precise these limits further down). 
Note that, as mentioned above, it is not immediately clear that the variance of $\bar{D}_{T,N}(q_0)$ doesn't diverge as $N \to \infty$. We provide in section \ref{subsection:boundedness_barD_TN} below conditions under which $\bar{D}_{T,N}(q_0)$ remains finite as $N \to \infty$.

The second term can be bounded by the supremum of the process $\{M_{2,T,N}(q, q_0) : q \in \mathcal{M}_Q \}$. This process is an empirical process generated by the sequence $(X(t,i))_{t,i}$, where we define that $X(t,i) := (C(t,i), L(t,i))$. So as to analyze this term, we prove a maximal inequality for such processes which holds under mixing conditions (here on the sequence $(X(t,i))_{t,i}$) and entropy conditions.  This maximal inequality will allow us to show the negligibility of $M_{2,T,N}(\widehat{q}^*_{T,N}, q)$ in front of the first term. 

We will discuss the negligibility of the remainder term $R(\widehat{q}^*_{T,N}, q_0)$ under convergence rate conditions on $\widehat{q}^*_{T,N}$

\subsection{Boundedness of the canonical gradient}\label{subsection:boundedness_barD_TN}

We can rewrite $\bar{D}_{T,N}(q_0)$ as 
\begin{align}
    \bar{D}_{T,N}(q_0)(l,c) = \sum_{s=1}^\tau \frac{1}{N} \sum_{j=1}^N \frac{h_{0,s,j}^*}{\bar{h}_{0,T,N}}(c) \widetilde{D}_{s,j,N}(l,c),
\end{align}
with 
\begin{align}
    \widetilde{D}_{s,j,N}(q)(l,c) := \sum_{k=1}^N E_{q,g^*} \left[ Y(k) \mid L(s,j) =l, C(s,j) = c \right] - E_{q, g^*} \left[ Y(k) \mid C(s,j) = c \right].
\end{align}
A sufficient condition for $\bar{D}_{T,N}(q_0)$ to remain bounded as $N \to \infty$ is that the terms $\widetilde{D}_{s,j,N}(q)$ themselves remain bounded as $N \to \infty$. It is immediate to observe that for $s = \tau$, since $L(s,j) \indep L(\tau, k) \mid C(s,j)$, and since $Y(k)$ is a a component of $L(\tau, j)$, we must have that $$E_{q,g^*} \left[ Y(k) \mid L(s,j) = l, C(s,j) = c \right] - E_{q, g^*} \left[ Y(k) \mid C(s,j) = c \right] = 0$$ for every $ j \neq k$, and therefore $\|\widetilde{D}_{\tau,j,N}(q) \|_\infty \leq 1$.

If we don't make any assumption on $g^*$, and that we just assume that under $g^*$, $A(s,1), \allowbreak \ldots, \allowbreak A(s,N)$ are conditionally independent given $A(t)^-$, but can a priori depend on the entire past $A(t)^-$, then, if $j \neq k$, we don't have the same kind of conditional independence between $Y(k)$ and $L(s,j)$,  for $s \leq \tau -1$ as we have in the case $s = \tau$. As a result, we don't have the same cancellations as in the case $s=\tau$. Intuitively, for $\| \widetilde{D}_{s,j,N}(q)\|_\infty$ not to diverge as $N \to \infty$, we need some measure of association between $Y(k)$ and the nodes $O(s,1), \ldots, O(s,N)$ to remain controlled in some sense. A natural measure of association that can be used to formulate rigorously this requirement is the classical notion of $\varphi$-mixing coefficient (see \cite{bradley2005} for a survey of usual mixing coefficients), which we restate here in terms of densities.

\begin{definition}[$\varphi$-mixing]\label{def:varphi-mixing} For any two random variables $(X,Y) \sim P$, we define the $\varphi$-mixing coefficient between $X$ and $Y$ as
\begin{align}
    \varphi_P(X,Y) := \sup \{ |p_{Y \mid X}(y \mid x) - p_X(x)| : y, x, \text{ such that } p_{X}(x) > 0 \},
\end{align}
where $p_{Y\mid X}$ and $p_X$ are the conditional densities of $Y$ given $X$ and the marginal density of $X$ w.r.t. an appropriate known dominating measure.
\end{definition}

We now provide a generic condition under which $\|\widetilde{D}_{s,j,N}(q)\|_\infty$, and therefore $\|\bar{D}_{T,N}(q)\|_\infty$ are controlled. We introduce the short-hand notation $\varphi_{q,g^*}$ for $\varphi_{P^{\tau,N}_{g^*}}$, where we recall that $P^{\tau,N}_{g^*}$ is the G-computation formula obtained from $P^{T,N}$.

\begin{assumption}\label{assumption:phi_mixing}
    Suppose that there exists $\bm{\varphi} < \infty$ such that, for any $s \in [\tau]$ and $k \in [N]$,
    \begin{align}
        \sum_{j=1}^N \varphi_{q, g^*}(Y(k) \mid O(s,j)) \leq \bm{\varphi}.
    \end{align}
\end{assumption}

Under assumption \ref{assumption:phi_mixing}, it is easy to show the following result.

\begin{lemma}\label{lemma:bound_tildeD_phi_mixing}
Suppose that assumption \ref{assumption:phi_mixing} holds. Then, $\|\widetilde{D}_{s,j,k}(q)\|_\infty \leq 2 \bm{\varphi}$.
\end{lemma}
A sufficient condition for $\bar{D}_{T,N}(q_0)$ to be bounded is then a bound on the marginal density ratios.
\begin{assumption}[Marginal density ratios bound]\label{assumption:marginal_density_ratios_bound}
    Suppose that there exists $B > 0$ such that
    \begin{align}
        \left\lVert h_{0,s,j}^*/\bar{h}_{0,T,N} \right\rVert_\infty \leq B
    \end{align}
    for every $s \in [\tau]$ and every $j \in [N]$.
\end{assumption}

\begin{lemma}\label{lemma:bound_barD_phi_mixing}
Suppose that assumptions \ref{assumption:phi_mixing} and \ref{assumption:marginal_density_ratios_bound} hold. Then $\| \bar{D}_{T,N}(q_0) \|_\infty \leq 2 \tau B \bm{\varphi}$.
\end{lemma}

While it might be hard to check that assumption \ref{assumption:phi_mixing} holds in practice, we see the value of it and of lemmas \ref{lemma:bound_tildeD_phi_mixing} and \ref{lemma:bound_barD_phi_mixing} in that they show explicitly the nature of a condition that is sufficient for $\bar{D}_{T,N}$ to remain bounded, that is a mixing condition controlling the level of association within the graph, across time points and individuals. We now discuss a few concrete examples where we can directly show that assumption \ref{assumption:phi_mixing} holds, or where we think it is reasonable to suppose that it holds.

\paragraph{Example 1.} If $P_0^{T,N} \in \mathcal{M}_2^{T,N}$, then, under the G-computation formula distribution $P^{\tau,N}_{0, g^*}$, any two distinct trajectories $\bar{O}(\tau,i)$ and $\bar{O}(\tau,j)$  are independent.

Therefore $\sum_{j=1}^N \varphi_{q,g^*}(Y(k) \mid O(s,j)) =\varphi_{q,g^*}(Y(k) \mid O(s,k))  \leq 1$, and thus $\|\widetilde{D}_{s,j,N}\|_\infty \leq 1$. (We can also directly check that all terms except the $k$-th one cancel out in $\widetilde{D}_{s,j,N}$). 

\paragraph{Example 2.} Suppose now that $g^*_{s,j}(A(s,j) \mid A(s)^-) = g^*_{s,j}(A(s,j) \mid C^*_A(s,j))$, with $C_A^*(s,j) \allowbreak := c_{A(s,j)}^*( \{ \bar{O}(s,j) : j \in F_A(s,j) \})$, where $F_A(s,j)$ is a set of at most $N_0$, individuals, including $j$ itself, and where $c^*_{A(s,j)}$ is a known summary function. In words, we are supposing here that the treatment decision under $g^*$ for individual $j$ at time $s$ depends only on the history up to $s-1$ of $j$ and of a set of at most $N_0$ individuals. Then any $Y(k)$ is associated with at most $N_0$ nodes from time point $\tau-1$, which are then in turn each associated with at most $N_0$ nodes from time point $\tau - 2$, and so on. Therefore, any $Y(k)$ is associated with at most $N_0^{\tau-s}$ nodes from time point $s$, and therefore 
$\sum_{j=1}^N \varphi_{q,g^*}(Y(k) \mid O(s,j))$ has at most $N_0^{\tau-s}$ non zero terms, which implies that $\|\widetilde{D}_{s,j,N}(q_0) \|_\infty \leq N_0 \tau^{\tau  -s }$, and thus $\|\bar{D}_{T,N}(q) \|_\infty \leq B \tau N_0^\tau$.

While this upper bound can quickly explode as $\tau$ gets large, this shows that for fixed $\tau$, the variance does not depend on $N$, and that therefore, under mixing conditions on the sequence $(O(t,i))_{t,i}$ that we study in the following subsection, the asymptotic variance of our estimators can scale as $N^{-1}$.

\paragraph{Example 3.} Suppose now that $g^*_{s,j}(A(s,j) \mid A(s)^-) = g^*_{s,j}(A(s,j) \mid \theta_N(s-1))$, where $\theta_N(s-1) = \frac{1}{N} \sum_{j=1}^N f(L(s,j))$ for a certain $f$. As $\theta_N(s-1)$ concentrates, and should be almost constant for large $N$, we expect that since treatment decisions depend on the past only through this almost constant $\theta_N(s-1)$, treatment assignment dependence on the past should not introduce too much dependence between units. In this situation, we conjecture that most of the dependence within the graph happens through the dependence of nodes $L(s,j)$ on the contacts  $F_L(s,j)$ of $j$. We have seen in the previous example that if this is the main source of dependence, we should have $\|\bar{D}_{T,N}(q_0)\| \lesssim \tau B N_0^\tau$, provided that $|F_L(s,j)| \leq N_0$ for every $s$ and $j$.

In our infectious disease example, the setting described here can model the situation where the intervention $g^*$ is to restrict, depending on the global infection rate $\theta_N(s)$, the set of individuals any individual $j$ is allowed to meet 

\subsection{Weak invariance principle for the martingale term}

It is immediate to observe that $E_{q_0, h_{0,t,i}}[\bar{D}_{T,N}(q_0)(L(t,i), C(t,i) \mid C(t,i) ] = 0 $, and therefore, $x T  M_{1,xT,N}(q_0)$ is the sum of a martingale difference sequence. We will analyze the weak convergence properties of the process $\{ M_{1,xT,N}(q_0) : x \in [0,1] \}$ via a classic functional central limit theorem for martingale triangular arrays, which we recall below.

\begin{theorem}[Theorem 3.2 \cite{mcleish1974}] \label{thm:mcleish_FCLT} Suppose that $\{X_{n,i} : 1 \leq i \leq n \}$ is a martingale difference array, and $(k_n)$ is a sequence of non-decreasing, right continuous, integer valued functions,  such that for every $n$, $k_n(0) = 0$. Let, for any $x \in [0,1]$ $W_n(x) := \sum_{i=1}^{k_n(x)} X_{n,i}$. Suppose that, for every $x \in [0,1]$
\begin{align}
    \max_{i \leq k_n(x)} |X_{n,i}| \xrightarrow{L_2} 0, \label{eq:cond1_McLeish}
\end{align}
and
\begin{align}
    \sum_{i=1}^{k_n(x)} X_{n,i}^2 \xrightarrow{P} x. \label{eq:cond2_McLeish}
\end{align}
Then $W_n \xrightarrow{d} W$ in $\mathbb{D}([0,1])$.
\end{theorem}
We will apply the above result by rewriting $M_{1,xT,N}(q_0)$ as the sum of a martingale difference triangular array, as we will make explicit in the proof. A key step ahead of applying theorem \ref{thm:mcleish_FCLT}  is to show that the variance under $P_{q_0, h_{0,t,i}}$ of $\bar{D}_{T,N}(L(t,i), C(t,i))$ stabilizes as $T,t \to \infty$. We provide below a set of conditions under which it is the case.

\begin{assumption}\label{assumption:cvgence_inverse_marginals_bounded_marginal_ratios} For every $N$, there exists $h_{0, \infty, N}$ such that, for every $t,i$, $\| \bar{h}_{0,T,N}^{-1} - h_{0, \infty, N}^{-1} \|_{2, h_{0,t,i}} \to 0$ as $T \to 0$, and there exists $B > 0$ such that $\|h^*_{0,s,j} / h_{0, \infty, N}\|_\infty \leq B$. 
\end{assumption}

\begin{assumption}\label{assumption:ergodicity}
For every $i$, $C(t,i) \xrightarrow{d}C_\infty \sim h_{0, \infty, N}$ as $t \to \infty$.
\end{assumption}

The first part of assumption \ref{assumption:cvgence_inverse_marginals_bounded_marginal_ratios}, and assumption \ref{assumption:ergodicity} are ergodicity/mixing conditions. Under these assumptions, we can show the following result on the limit of $$\Var_{q_0, h_{0,t,i}}(\bar{D}_{T,N}(q_0)(L(t,i), C(t,i)).$$

\begin{lemma}[Stabilization of the variance of the main term of the EIF]\label{lemma:stablization_variance_EIF}
Suppose that assumptions \ref{assumption:phi_mixing}, \ref{assumption:marginal_density_ratios_bound}, \ref{assumption:cvgence_inverse_marginals_bounded_marginal_ratios} and \ref{assumption:ergodicity} hold. 
Denote 
\begin{align}
    \bar{D}_{0, \infty, N}(l,c) := \sum_{s=1}^\tau \frac{1}{N} \sum_{j=1}^N \frac{h^*_{0,s,j}}{h_{0,\infty,N}}(c) \widetilde{D}_{s,j,N}(l,c),
\end{align}
and
\begin{align}
    \sigma_{0,\infty,N}^2 := \Var_{q_0, h_{0, \infty, N}} \left( \bar{D}_{0, \infty, N}(L_\infty, C_\infty) \right),
\end{align}
Then $ \sigma_{0,\infty,N} < \infty$, and 
\begin{align}
    \Var_{q_0, h_{0,t,i}} \left( \bar{D}_{T,N}(q_0)(L(t,i), (C(t,i)) \right) \to  \sigma_{0,\infty,N}^2, \text{ as } T,t \to \infty.
\end{align}
\end{lemma}

For any $t$, $i$, let $X(t,i) := (C_L(t,i), L(t,i))$.
A key requirement for our analysis of the process $\{M_{1,xT, N}(q_0) : x \in (0,1] \}$ is an $\alpha$-mixing condition on the sequence $(X(t,i))_{t,i}$. We first recall the notion of $\alpha$-mixing. We give here a definition based on theorem 4.4 in \cite{bradley2007}.
\begin{definition}[$\alpha$-mixing]
Consider a couple of random variables $(X,Y) \sim P$, with marginals $P_X$ and $P_Y$ and domains $\mathcal{X}$ and $\mathcal{Y}$. The $\alpha$-mixing coefficient between $X$ and $Y$ is defined as 
$$\alpha_P(X,Y) := \sup \left\lbrace \frac{\Cov(f(X), g(Y))}{\|f\|_\infty \|g\|_\infty}, f:\mathcal{X} \to \mathbb{R}, g : \mathcal{Y} \to \mathbb{R}, \|f\|_\infty < \infty, \|g\|_\infty < \infty \right\rbrace.$$
\end{definition}
We state our $\alpha$-mixing condition below.

\begin{assumption}[$\alpha$-mixing]\label{assumption:rho_mixing}
It holds that 
$$\sum_{t_1,t_2 \in [T]} \sum_{i_1, i_2 \in [N]} \alpha_P(X(t_1, i_1), X(t_2, i_2)) = o(TN).$$
\end{assumption}

\begin{theorem}[Weak convergence of the martingale term $M_1$ ]\label{thm:weak_convergence_M1xTN} 
Suppose that assumptions 
\ref{assumption:phi_mixing}, \ref{assumption:marginal_density_ratios_bound}, \ref{assumption:cvgence_inverse_marginals_bounded_marginal_ratios}, \ref{assumption:ergodicity} and \ref{assumption:rho_mixing} hold. Then, for any fixed $N$, as $T \to \infty$,
\begin{align}
    \{ M_{1,xT,N}(q_0) : x \in [0,1] \} \xrightarrow{d} \sigma_{0,\infty,N} W
\end{align}
on the set $\mathbb{D}([0,1])$ of cadlag functions on $[0,1]$, where $ \sigma_{0,\infty,N} \leq C$ for some $0 < C < \infty$ that does not depend on $N$, and where $W$ is a standard Wiener process.
\end{theorem}

\subsection{Analysis of the empirical process term under mixing conditions}

Recall that we defined $X(t,i) := (C_L(t,i), L(t,i))$. The process $\{M_{2,T,N}(q,q_0) : q \in \mathcal{Q} \}$ is an empirical process generated by the sequence of dependent observations $(X(t,i))_{t,i}$. 

For there to be a hope of controlling the deviations of this process from its mean, we must impose conditions on the amount of dependence between terms of the sequence $(X(t,i))_{t,i}$. As in the analysis of the term $M_{1,xT,N}$, we impose mixing conditions. Let $(\widetilde{X}(k))_k$ be the single-index sequence obtained by reordering the terms of the double-index sequence $(X_{t,i})_{t,i}$ in colunm order, that is $(\widetilde{X}_k)_k$ is the sequence
\begin{align}
    X(1,1),\ldots,X(1,N),\ldots,X(T,1),\ldots,X(T,N).
\end{align}
We define $(\widetilde{C}_L(k))_k$ and $(\widetilde{L}(k))_k$ similarly. We state our mixing conditions in terms of the sequence $(\widetilde{X}(k))_k$.

\begin{assumption}[Geometric $\alpha$-mixing]\label{assumption:main_text_exp_alpha_mixing} There exists $c > 0$ such that the $\alpha$-mixing coefficients of $(\widetilde{X}(k))_{k \geq 1}$ satisfy $\alpha(n) \leq \exp(-cn)$.
\end{assumption}

The next assumption is a $\rho$-mixing condition on the sequence $(\widetilde{X}(k))$. We state below the definition of $\rho$-mixing. We refer the reader to \cite{bradley2005} for more details on mixing coefficients.
\begin{definition}[$\rho$-mixing]
Consider a couple of random variables $(X,Y) \sim P$, with marginals $P_X$ and $P_Y$ The maximum correlation coefficient between $X$ and $Y$ is defined as $\rho_P(X,Y) := \sup \{\mathrm{Corr}(f(X), g(Y)) : f \in L_2(P_X), g \in L_2(P_Y)\}$.
\end{definition}

\begin{assumption}[$\rho$-mixing]\label{assumption:main_text_finite_rho_mixing} The $\rho$-mixing coefficients of $(\widetilde{X}(k))$ have finite sum, that is $\sum_{n=1}^\infty \rho(n) := \bm{\rho} < \infty$.
\end{assumption}

The main result of this section is an almost sure equicontinuity result, which will give us that
$\sqrt{NT} M_{2,N,T}(q_{N,T}, q_0) = o(1)$ almost surely. As for similar equicontinuity results (see e.g. \cite{vdV_Wellner96}) for i.i.d. empirical processes, we require two types of conditions: (1) we need that the individual terms of $M_{2,T,N}(q_{N,T}, q_0)$ converge to zero, in some sense to be made precise further down, and (2) we need a Donsker-like condition on the complexity of the class $\{\bar{D}_{T,N}(q, q_0) : q \in \mathcal{Q} \}$. 

While equicontinuity results for empirical processes usually give a convergence in probability guarantee, we prove an almost sure convergence result. Almost sure convergence offers control over the entire realization of the sequence $ (M_{2,T,N}(q_{N,T}, q_0))_{N,T}$, which we need in section \ref{section:adaptive_stopping_rules} to design an adaptive stopping rule. As we work in a more challenging setting (mixing sequences v.s. i.i.d. sequences), and as we prove stochastic convergence in a stronger sense, we need stronger versions of the Donsker condition than in the classical equicontinuity results for empirical processes (those of \cite{vdV_Wellner96} for example). In particular, while the classical results require convergence of some type of $L_2$ norm of the difference $\bar{D}_{T,N}(\widehat{q}_{T,N}) - \bar{D}_{T,N}(q_0)$, we require convergence of this difference in $\|\cdot\|_\infty$ norm. Furthermore, while in the classical results the type of stochastic convergence required is convergence in probability, here we need a form of stochastic convergence slightly stronger than almost sure convergence.

We formulate precisely our convergence requirement in the following assumption.

\begin{assumption}\label{assumption:uniform_control_can_gdt_sequence} There exists a sequence of positive numbers $(a_T)$ satisfying $a_T^{-2} (\log T)^2 / \sqrt{T} = o(1)$, and $a_T^\nu \log T = o(1)$ for any $\nu > 0$, such that 
\begin{align}
    \forall \epsilon > 0, \ \exists C(\epsilon) > 0, \ P\left[ \forall n \geq 1, \left\lVert \bar{D}_{T,N}(\widehat{q}_{T,N}) - \bar{D}_{T,N}(q_0) \right\rVert_\infty \leq C(\epsilon) a_n \right] \geq 1 - \epsilon.
\end{align}
\end{assumption}

We introduce in section \ref{section:Lipschitz_HAL_class_and_nuisance_estimation} a large nonparametric class of $d$-variate functions $\mathcal{F}_d$, which is such that, in our dependent data setting, any empirical risk minimizer $\widehat{q}_{T,N}$ over any $\mathcal{Q} \subseteq \mathcal{F}_d$ satisfies an exponential deviation bound of the form $P\left[\|\widehat{q}_{T,N} - q_*\|_\infty \gtrsim n^{-\beta} + x \right] \lesssim \exp(- C (n x^\gamma)^\nu)$, with $\beta, \gamma, \nu > 0$, $1 - \gamma \beta >0$, where $q_*$ is a population risk minimizer over $\mathcal{Q}$. If the true transition density $q_0$ lies in our nonparametric class $\mathcal{Q}$, and if $\left\lVert \bar{D}_{T,N}(q) - \bar{D}_{T,N}(q_0) \right\rVert_\infty \allowbreak \lesssim \|q - q_0\|_\infty$, then it is straightforward to show that assumption \ref{assumption:uniform_control_can_gdt_sequence} holds.

We now present our Donsker-like condition. Suppose that, for any $k$, the distribution of $\widetilde{C}(k)$ admits density $\widetilde{h}_k$ w.r.t. the Lebesgue measure. The density w.r.t the Lebesgue measure of $\widetilde{X}(k)$ is then $q_0 \widetilde{h}_k$. Let $\mathcal{X}$ be such that, for any $t$ and any $i$, $X(t,i)$ takes values in $\mathcal{X}$. Let $\sigma$ be the norm defined, for any $f: \mathcal{X} \to \mathbb{R}$ by
\begin{align}
    \sigma(f) := \sup_{i \geq 1} \|f\|_{2,q_0, \widetilde{h}_i} \sqrt{1 +  2 \bm{\rho}}.
\end{align}
Our Donsker-like condition is a bound on the bracketing entropy in $\sigma$ norm of the canonical gradient class.
\begin{assumption}[Donsker condition for the canonical gradient class]\label{assumption:Donsker_can_gdt_class}
Let $\mathcal{D}_{T,N} := \{ \bar{D}_{T,N}(q) : q \in \mathcal{Q} \}$.
There exists $p \in (0,2)$ such that 
\begin{align}
    \log N_{[\,]}(\epsilon, \mathcal{D}_{T,N}, \sigma) \lesssim \epsilon^{-p}.
\end{align}
\end{assumption}

We show in section \ref{section:Lipschitz_HAL_class_and_nuisance_estimation} that the nonparametric function class $\mathcal{F}_d$ we mentioned above satisfies $\log N_{[\,]}(\epsilon, \mathcal{F}, \sigma) \lesssim \epsilon^{-1} |\log (\epsilon)|^{2d - 1}$ under mild conditions. Therefore, if the canonical gradient class $\mathcal{D}_{T,N}$ is included in $\mathcal{F}_{d'}$ for some $d' \geq 1$, then assumption \ref{assumption:Donsker_can_gdt_class} holds under the same mild conditions.

We can now state our equicontinuity result.

\begin{theorem}[Asymptotic equicontinuity of the canonical gradient process]\label{thm:equicont_can_gdt_process}
Suppose that assumptions \ref{assumption:main_text_exp_alpha_mixing}, \ref{assumption:main_text_finite_rho_mixing}, \ref{assumption:uniform_control_can_gdt_sequence} and \ref{assumption:Donsker_can_gdt_class} hold. Then
\begin{align}
    \sqrt{NT} M_{2,T,N}(\widehat{q}_{T,N}, q_0) = o(1) \text{ a.s. as } T, N \to \infty.
\end{align}
\end{theorem}

\begin{proof}[Proof of theorem \ref{thm:equicont_can_gdt_process}] The proof is a direct consequence of our generic equicontinuity result, theorem \ref{thm:equicont_weakly_dep_EP} in the appendix.
\end{proof}

\begin{remark}
It might seem surprising to the reader familiar with proofs of equicontinuity results and maximal inequalities for empirical processes that, while the Donsker condition only requires control of the entropy w.r.t. the norm $\sigma$, which is an $L_2$ norm, we need convergence of $\|\bar{D}_{T,N}(\widehat{q}_{T,N}) - \bar{D}_{T,N}(q_0) \|_\infty$ in a norm a strong as the sup norm. Indeed, in the usual case where $Z_1,\ldots,Z_n$ are i.i.d. random draws from a distribution $P$ taking values in a set $\mathcal{Z}$, if $\mathcal{F}$ has square integrable bracketing entropy w.r.t. $L_2(P)$, for a process of the form $n^{-1} \sum_{i=1}^n f_n(Z_i) - \int f_n(z)dP(z)$ to be $o_P(n^{-1/2})$, it suffices that the $L_2(P)$ norm of $f_n$ converges to zero in probability.

We discuss in subsection \ref{subsection:equicont_weakly_dep_EP} in the appendix why, unlike in the i.i.d. setting, in the weakly dependent case we consider here, convergence in $L_2$ norm wouldn't suffice given the technical tools that we have, and why we do need convergence in $\|\cdot\|_\infty$ norm.
\end{remark}


\subsection{Weak invariance principle for our TMLE}

\begin{theorem}[Weak invariance principle for our TMLE]\label{thm:WIP_TMLE}
Suppose that the assumptions of theorems \ref{thm:weak_convergence_M1xTN} and \ref{thm:equicont_can_gdt_process} are satisfied, and that $R(\widehat{q}_{T,N}, q_0) = o((NT)^{-1/2})$ almost surely. We then have that the process
\begin{align}
    \left\lbrace t \sqrt{T N} \sigma_{0,\infty,N}^{-1} \left(\widehat{\Psi}_{t T,N} - \Psi(P_0^{t T,N}) \right) : t \in [0,1] \right\rbrace
\end{align}
converges weakly in $\mathbb{D}([0,1])$ to a Wiener process $W$. 
\end{theorem}

\section{A nonparametric function class, nuisance estimation, and the canonical gradient class}\label{section:Lipschitz_HAL_class_and_nuisance_estimation}

We first present a generic function class, which we then use as a statistical model for nuisance estimation and canonical gradient modelling.

\subsection{The function class}\label{subsection:the_Lipschitz_HAL_class}

Consider a bounded Euclidean set $\mathcal{X}$. Without loss of generality, we will suppose that $\mathcal{X} = [0,1]^{d_1}$, the unit hypercube in $\mathbb{R}^{d_1}$. For $M > 0$, let $\mathcal{F}_{0,M}$ be the class of real-valued cadlag functions on $\mathcal{X}$, with sectional variation norm (also called Hardy-Krause variation) no larger than $M$, and, for $L > 0$, let $\mathcal{F}_{1,M,L}$ be the class of functions in $\mathcal{F}_{0,M}$ that are $L$-Lipschitz. The classes $\mathcal{F}_{0,M}$ and $\cup_{M > 0} \mathcal{F}_{0,M}$ have been proposed as statistical models in several past articles \citep{vdL-generally_eff_TMLE2017, fang2020multivariate, bibaut2019fast}. We refer to these works for the rigorous definition of the notion of sectional variation norm. For the present purpose, it will suffice to say that the sectional variation norm is a multivariate extension of the 1-dimensional notion of total variation of a real-valued function.

\paragraph{Statistical properties of $\mathcal{F}_{0,M}$.} This class of functions present several attractive properties as a statistical model. \cite{bibaut2019fast} have shown that its bracketing entropy is well controlled, which will prove useful in our problem. We recall here the formal result on bracketing entropy from \cite{bibaut2019fast}.

\begin{proposition}[Proposition 2 in \cite{bibaut2019fast}]\label{prop:bkting_entropy_HAL}
Consider $\mathcal{F}_{0,M}$ as defined above. For any $r \geq 1$, $\epsilon > 0$, it holds that 
\begin{align}
    \log N_{[\,]}(\epsilon, \mathcal{F}_{0,M}, L_r(\mu)) \lesssim M \epsilon^{-1}  \left\lvert \log(M / \epsilon) \right\rvert^{2(d_1-1)},
\end{align}
where we have absorbed a constant depending on the dimension $d_1$ and on $r$ in the $``\lesssim''$ notation, and where $\mu$ is the Lebesgue measure.
\end{proposition}
Notice that the entropy depends on the dimension only through the log factor. As a result, even in high dimensions this class remains Donsker, and rates of convergence of empirical risk minimizers (ERMs) over it remain relatively fast. Unlike other popular nonparametric function classes such as Holder classes, $\mathcal{F}_{0,M}$ doesn't impose local smoothness restrictions. Rather, it only places a bound on a global measure of variation, the sectional variation norm, thus allowing for function having different degrees of smoothness or roughness at different regions of their domains. As a result, from a bias-variance trade-off perspective, when one increases $M$ by some amount, an ERM estimator will ``spend'' the additional allowed variation in the areas of the domain where it most improves the fit, while only impacting the entropy loglinearly. While this might not be a perfectly rigorous comparison, note that, Holder classes $H(M,\beta)$ have entropy depending on $\epsilon$ as $\epsilon^{d/ \beta}$, and therefore decreasing $\beta$ so as to reduce bias has a steep entropy price.

We believe that since the nonparametric model $\cup_{M > 0} \mathcal{F}_{0,M}$ only assumes a form of piecewise continuity and that the sectional variation norm is not infinite, using it a statistical model for components of the data-generating distributing amounts to a mild assumption. Our guess is that functions that do not satisfy these requirements are essentially pathological functions $x \mapsto f(x)$ that oscillate increasingly fast as $x$ approaches some value or region. \cite{benkeser2016} have shown with extensive simulations that ERMs over $\mathcal{F}_{0,\widehat{M}}$, with $\widehat{M}$ chosen by cross-validation, perform on par with Random Forests and Gradient Boosting Machines, thereby confirming that $\cup_{M > 0} \mathcal{F}_{0,M}$ is a realistic statistical model in most practical settings.

\paragraph{Computational properties.} \cite{fang2020multivariate} have shown that ERMs over $\mathcal{F}_{0,M}$ can be computed as the solution of a LASSO problem over at most $(n e / d)^d$ distinct basis functions, where $n$ is the sample size. \cite{vdL-generally_eff_TMLE2017} has proposed an alternative set of basis functions of cardinality $n 2^d$, which, although it can be shown to not always be sufficient to represent the ERM, leads to very good practical performance.

\paragraph{Properties of $\mathcal{F}_{1,M,L}$.} Introducing the additional assumption that the functions in our model are Lipschitz allows to bound the supremum norm of a function in terms of its $L_19(\mu)$ norm, as shown by the following lemma. We owe this result to Iosif Pinelis, who proved it as an answer to a question of the first author on MathOverflow \citep{pinelis2020-MO_answer}.

\begin{lemma}\label{lemma:sup_norm_bound_Lipschitz}
There exists $\eta(d,L) > 0$ and $C(d, L) > 0$ such that, for any $f$ is a $d$-variate, real-valued $L$-Lipschitz function such that $\|f\|_{1,\mu} \leq \eta(d,L)$, we have
$\|f\|_\infty \leq C(d,L) \|f\|_{1,\mu}^{1/(d+1)}$, with $\mu$ the Lebesgue measure.
\end{lemma}
Unlike in the i.i.d. setting, in our mixing data sequence setting, we will need to be able to show that the supremum norm of some functions converge to zero at a certain rate. We refer the interested reader to the proofs of the results of the next two subsections for more detail on these technical questions.

\subsection{Nuisance estimation}\label{subsection:nuisance_estimation}

The efficient influence function expression makes appear the nuisance parameters $q$, $(\phi_{s,j})$, $(h_{s,j}^*)$ and $\bar{h}_{N,T}$. The latter are functions of $q$ and can be computed from an estimate thereof via Monte-Carlo integration, as discussed in \cite{laan_chambaz_lendle2018} in the case $N=1$. The key statistical challenge is then the estimation of the true conditional density $q_0$.

We propose to estimate $q_0$ by a maximum likelihood estimator over the subset of functions of $\mathcal{F}_{1,M,L}(\mathcal{X})$, with $\mathcal{X} := \mathcal{C} \times \mathcal{O}$, that are conditional probability density functions $(c,o) \mapsto q(o \mid c)$, that is over the set
\begin{align}
    \mathcal{Q}_{M,L} := \left\lbrace q \in \mathcal{F}_{1,M,L} : \forall c \int q(o \mid c) do = 1 \text{ and } q(\cdot \mid c) \geq 0 \right\rbrace. 
\end{align}
In practice, $M$ and $L$ should be chosen by cross validation. As there will be no ambiguity in the rest of this section, we use the notation $\mathcal{Q}$ instead of $\mathcal{Q}_{M,L}$. In this section too, we work with the reordered single-indexed sequence $(\widetilde{O}(k))$ as defined in the previous section.

For any conditional density $q:(o,c) \mapsto q(o \mid c)$, let $\ell(q)(c,o) := - \log q(o \mid c)$ be the log-likelihood loss for $q$, and let
\begin{align}
    \widehat{R}_n(q) := \frac{1}{n} \sum_{i=1}^n \ell(q)(\widetilde{C}(k), \widetilde{O}(k)) \qquad \text{and} \qquad R_{0,n}(q) := \frac{1}{n} \sum_{i=1}^n E[\ell(q)(\widetilde{C}(k), \widetilde{O}(k))].
\end{align}

Let $\widehat{q}_n \in \arg\min_{q \in \mathcal{Q}} \widehat{R}_n(q)$ and $q_n \in \argmin_{q \in \mathcal{Q}} R_{0,n}(q)$ be a maximum likelihood estimator, and a maximizer over $\mathcal{Q}$ of the population log likelihood. We analyze $\widehat{q}_n$ using our generic result for ERMs under mixing sequences, theorem \ref{thm:exp_bound_for_ERM} in the appendix. We need the following assumptions.

\begin{assumption}[Lower bound on the population MLE]\label{assumption:lower_bound_pop_MLE} There exists $\delta$ independent of $n$ such that $\inf_{c,o \in \mathcal{C} \times \mathcal{O}} q_n(o \mid c) \geq \delta$.
\end{assumption}

\begin{assumption}\label{assumption:ratio_q0_qn}
There exists $M_1 > 0$ independent of $n$ such that $\| q_0 / q_n \|_\infty \leq M_1$.
\end{assumption}

\begin{assumption}[Uniform boundedness of $(\widetilde{h}_i)_{i \geq 1}$]\label{assumption:unif_boundedness_h_i}
There exists $M_2 > 0$ such that $\sup_{i \geq 1} \|\widetilde{h}_i \|_\infty \leq M_2$.
\end{assumption}

\begin{assumption}\label{assumption:bound_h_i_over_hbar_n}
Denote $\bar{\widetilde{h}}_n := n^{-1} \sum_{i=1}^n \widetilde{h}_i$. There exists $M_3  >0$ independent of $n$ such that 
\begin{align}
    \sup_{i \geq 1} \left\lVert \widetilde{h}_i / \bar{\widetilde{h}}_n \right\rVert_\infty \leq M_3.
\end{align}
\end{assumption}

\begin{theorem}[High probability bound on the MLE of $q_0$]\label{thm:hp_bound_MLE_q0}
Suppose that assumptions \ref{assumption:main_text_exp_alpha_mixing}, \ref{assumption:main_text_finite_rho_mixing}, \ref{assumption:lower_bound_pop_MLE}, \ref{assumption:ratio_q0_qn}, \ref{assumption:unif_boundedness_h_i} and \ref{assumption:bound_h_i_over_hbar_n} hold. Then, letting $\alpha := 1/(d+1)$, it holds that, for every $x > 0$, with probability at least $1 - 2 e^{-x}$, that 
\begin{align}
    \sigma\left(\widehat{q}_n - q_n \right) \lesssim n^{-\frac{1}{4 - 2 \alpha}} + \log n \sqrt{\frac{x}{n}} + \left( \log n \right)^{\frac{2}{2-\alpha}} \left(\frac{x}{n} \right)^{\frac{1}{2 - \alpha}}.
\end{align}
\end{theorem}

\subsection{Canonical gradient}

As argued in subsection \ref{section:Lipschitz_HAL_class_and_nuisance_estimation}, we think that assuming that components of the data-generating distribution $P_0^{T,N}$ lie in $\mathcal{F}_{1,M,L}$ for some $M, L > 0$ is a relatively mild modelling assumption. We therefore assume that 
\begin{align}
    \left\lbrace \bar{D}_{T,N}(q) : q \in \mathcal{Q} \right\rbrace \subset \mathcal{F}_{1,M,L}. \label{eq:can_gdt_in_LHAL_class}
\end{align}
We conjecture that this actually automatically follows if $\mathcal{Q} \subset \mathcal{F}_{1,M,L}$ and we think that one could prove this using the usual arguments to prove bracketing numbers preservation results. However, this appears to be tedious, so we leave it to future work.
 Under \eqref{eq:can_gdt_in_LHAL_class}, assumption \ref{assumption:Donsker_can_gdt_class} holds. If we further assume that $\| \bar{D}_{T,N}(q_2) - \bar{D}_{T,N}(q_1) \|_\infty \lesssim \| q_2 - q_1 \|_\infty$, lemma \ref{lemma:sup_norm_bound_Lipschitz} and \ref{thm:hp_bound_MLE_q0} then imply that assumption \ref{assumption:uniform_control_can_gdt_sequence} holds.

\section{Adaptive stopping rules}\label{section:adaptive_stopping_rules}

In this section, we present an adaptive stopping rule for the test of the hypothesis $H_0 : \Psi(P_0^{T,N}) = 0$.  In practice, it is natural to consider a parameter of the form $\Psi(P_0^{T,N}) = E_{P_0^{T,N}}[Y^{g^*_1} - Y^{g^*_2}]$, for which the analysis follows trivially from the the analysis of the individual terms of the difference we have presented so far. (Note that $H_0$ doesn't actually depend on $T$ and $N$, since  $\Psi(P_0^{T,N})$ can be written as $\Psi^{(1)}(q_0)$, as pointed out in section \ref{section:problem_formulation}).  An adaptive stopping rule allows to reject the null hypothesis as soon as sufficient evidence has been collected, without the need to wait for a pre-specified sample size to be met. Since an adaptive stopping rule checks a a criterion at every time step, multiple testing considerations must be taken into account so as to make sure the type I error remains controlled. 

A typical approach to design a valid adaptive stopping rule is as follows. Say we want to ensure that type I error is no larger than $1-\alpha$. The key step is to construct a uniform-in-time $(1-\alpha)$-probability confidence band, that is sequence of confidence intervals$([\pm a_{\alpha,N}(T)])_{T \geq 1}$, such that, with probability $1-\alpha$, $\Psi(P_0^{T,N}) \in [\pm a_{\alpha,N}(T)]$ for every $T$. Then a natural stopping rule is to reject the null hypothesis at the earliest time $T$ such that $0 \not\in [\pm a_{\alpha,N}(T)]$.
A uniform-in-time confidence band is a feature of the joint distribution of the sequence of estimators $(\widehat{\Psi}_{N,T})_{T \geq 1}$. Theorem \ref{thm:WIP_TMLE} characterizes in an asymptotic sense the joint distribution of a process obtain from the finite sequence $(\widehat{\Psi}_{N,T})_{t=1}^T$ by rescaling it in time and in range: specifically, it shows that $\{ t \sqrt{T N } \sigma_{0, \infty, N}^{-1}(\widehat{\Psi}_{N,T} - \Psi(P_0^{T,N})) : t \in [0,1] \}$ converges weakly to a Wiener process. Since confidence bands for the Wiener process are well documented, we will be able to use this to construct an adaptive stopping rule.

Since our results on the joint distribution of the (rescaled process built from the) sequence of estimates are asymptotic, our procedure requires a certain burn-in period, that is we must enforce a minimum time point before which the procedure cannot reject. We now present formally our type I error guarantees for the procedure we described.

\begin{theorem}[Type I error of adaptive stopping]\label{thm:type_I_error_adaptive_stopping} Let $(a_\alpha(t) : t \in [0,1])$ be such that 
$P[\forall t \in [0,1], W(t) \in [\pm a_\alpha(t)] ] \geq 1$. Let $T_{\max}$ be the maximum number of time steps the experimenter is willing to run the trial. Let $t_0 \in [0,1]$ be such that $T_0 := t_0 T_{\max}$ is the duration of the burn-in period.

Let 
\begin{align}
    \tau(T_{\max}, t_0) := \min \left\lbrace T \geq T_0, \widehat{\Psi}_{T,N} \not\in \left[ \pm \sigma_{0,\infty,N} \frac{\sqrt{T_{\max} / T} a(T/ T_{\max})}{\sqrt{NT}}\right] \right\rbrace.
\end{align}

Suppose that the assumptions of theorem \ref{thm:WIP_TMLE} are satisfied. Then, under the null hypothesis $H_0 : \Psi(P_0^{T,N}) = 0$, it holds that
\begin{align}
    \lim_{T_{\max} \to \infty} P_0 \left[ \tau(T_{\max}, t_0) \leq T_{\max} \right] \geq 1 - \alpha,
\end{align}
that is the probability that the procedure rejects under the null is asymptotically no larger than the nominal level $\alpha$.
\end{theorem}

In practice, theorem \ref{thm:type_I_error_adaptive_stopping} teaches us that for reasonably large horizon $T_{\max}$ and burn-in period $T_0$, the procedure has type-I error approximately no larger than $1-\alpha$.

An alternative direction to construct an adaptive stopping rule would be to analyze the deviations of our estimator with uniform-in-time concentration bounds, such as the ones presented in \cite{howard2018timeuniform}, instead of using a limit theorem. We leave this direction for future research. We nevertheless point out that exact confidence bands/intervals obtained from concentration inequalities tend to be larger than approximate confidence bands/intervals obtained from FCLTs/CLTs. As a result, we conjecture that controlling exactly, rather than approximately the type I error by using concentration inequalities rather than limit theorems might cost a signicant loss of power.

\section{Learning the optimal design along the trial}\label{section:learning_optimal_design}

Consider a target parameter of the form $\Psi_\tau(q) := E_{q,g_2^*} Y - E_{q,g_1^*} Y$, where $Y$ is an outcome at time $\tau$, as defined earlier, and where $g_1^*$ and $g_2^*$ are known and fixed stochastic interventions. In the best arm identification example in the case where there are two arms, we would have $g_1^*(a \mid c^A) = \Ind(a = 1)$ and $g_2^*(a \mid c^A) = \Ind(a = 2)$, that is $g_1^*$ and $g_2^*$ are the deterministic interventions that always assign the same treatment. In the infectious disease example, $g_1^*$ and $g_2^*$ would be two different public health interventions, such as imposing that individuals wear a mask, or that they stay at home for except for a certain set of allowed activities. 

Suppose that we have a collection of candidate designs $g_1(q), \ldots, g_J(q)$ that are indexed by $q$. We would like to achieve the same asymptotic variance as we would if we had carried out the best design among $g_1(q_0), \ldots, g_J(q_0)$ from the beginning. Let us make this more formal. Making explicit that $h_{\infty, N}$ depend on $q$ and $g$, we will write $h_{\infty, N}(q,g)$. For every $k$, let 
\begin{align}
    \chi_k(q) := \Var_{h_{\infty, N}(q, g), q} \left( (\bar{D}^{g_1^*}_{T,N}(q,g) - \bar{D}_{T,N}^{g_2^*}(q,g))(C_\infty, L_\infty) \right)
\end{align}
be the asymptotic variance of the EIF under $g_k$. 

Given an estimator $\widehat{q}_{T,N}$ of $q_0$, we can compute (approximately by Monte-Carlo simulation for example) $\chi_k(\widehat{q}_{T,N})$, the plug-in estimator of $\chi_k(q_0)$. Let $k(T) := \argmin_{k \in [J]} \chi_k(\widehat{q}_{T,N})$. We define our adaptive design at $T$ as $g_{k(T-1)}(\widehat{q}_{T-1,N})$.

We now study heuristically the conditions under which this adaptive design is such that the TMLE of $\Psi_\tau(q_0)$ achieves the asymptotic variance $\chi_{k^*}(q_0)$, with $k^* = \argmin_{k \in [J]} \chi_k(q_0)$, that is the optimal asymptotic variance among the $J$ designs considered. 

Suppose that $\widehat{q}_{T,N}$ converges almost surely to $q_0$ and that $\chi_1(q_0), \ldots, \chi_J(q_0)$ are distinct. Then $\chi_k(\widehat{q}_{T,N})$ converges a.s. to $\chi_k(q_0)$, and therefore, with probability 1, $k(T) \neq k^*$ only a finite number of times. Therefore, we expect that $\| \bar{h}_{0,T,N} - h_{\infty, N}(q_0, g_{k^*}(q_0)) \|_1 = o(1)$, which in turns, if $\bar{h}_{0,T,N}$ and $h_{\infty, N}(q_0, g_{k^*}(q_0))$ are lower bounded away from zero implies that $\| \bar{h}_{0,T,N}^{-1} - h_{\infty, N}^{-1}(q_0, g_{k^*}(q_0)) \|_1 = o(1)$. Therefore, under the assumptions of lemma \ref{lemma:stablization_variance_EIF}, the variance of the the terms $\bar{D}_{T,N}^{g_1^*} - \bar{D}_{T,N}^{g_2^*}$ of the EIF should stabilize to $\chi_{k^*}(q_0)$, which under the assumptions of theorem \ref{thm:WIP_TMLE}, implies that under the adaptive design, the asymptotic variance of the TMLE must be $\chi_{k^*}(q_0)$.

\paragraph{Examples of candidate designs in the best arm identification example.} In the best arm identification example, in the case where there are only two arms and where $\tau = 1$ (that is the target is the ATE after one time step, starting from a known distribution of contexts), it is known that the optimal design is the so-called Neyman allocation design, defined as follows:
\begin{align}
    g(q_0)(a \mid c^A) := \frac{ \sigma_{q_0}(a,c^A) }{\sigma_{q_0}(1,c^A) + \sigma_{q_0}(2,c^A)},
\end{align}
where
\begin{align}
    \sigma_{q}^2(a,c^A) := \Var_q(Y(1) \mid A(1,1) = a, C^A(1,1) = c^A).
\end{align}
In words, the Neyman allocation designs assigns treatment $a$ with probability proportional to the standard deviation of the outcome conditional on $a$ and $c^A$. 

While we don't know whether this design is optimal design among all possible designs in the case $\tau > 1$, we conjecture it should be more efficient that the uniform design over treatment arms. In practice, we recommend considering a finite library of candidate designs including the Neyman allocation design. Other possible designs are the constant design with fixed probabilities for each arm.
\bibliographystyle{plainnat}
\bibliography{biblio}

\appendix

\section{Notation}

\subsection{Notation relative to the data}

\begin{align}
    A(t,i) :& \text{ treatment assigned to individual } i \text{ at } t,\\
    L(t,i) :& \text{ time varying covariates and outcomes  of individual } i \text { at } t,\\
    O(t,i) :=& (A(t,i), L(t,i),\\
    A(t) :=& (A(t,i) : i \in [N]) \\
    L(t) :=& (L(t,i) : i \in [N]) \\
    \bar{A}(t) =& (A(1), \ldots, A(t)),\\
    \bar{L}(t) :=& (L(1), \ldots, L(t)), \\
    \bar{O}(t) :=& (O(1), \ldots, O(t)),\\
    \bar{A}(t,i) :=& (A(s,i) : s \in [t]),\\
    \bar{L}(t,i) :=& (L(s,i) : s \in [t]),\\
    \bar{O}(t,i) :=& (O(s,i) : s \in [t]),\\
    O^{T,N} :=& (O(t,i) : t \in [T], i \in [N]).
\end{align}
Observe that $O^{T,N} = \bar{O}(T)$.

Data is observed in the order $A(1), L(1),\ldots, A(t), L(t)$. Within time points, we arbitrarily order data points by increasing index $i$, that is, we order individual observations as:
\begin{align}
    A(1,1),\ldots,A(1,N),L(1,1),\ldots,L(1,N),\ldots,A(T,1),\ldots,A(T,N),L(T,1), \ldots, L(T,N).
\end{align}
We refer to this ordering as the column ordering. We let $A(t,i)^-$ and $L(t,i)^-$ be the vectors of all observations that come before $A(t,i)$ and $L(t,i)$ in the column ordering, that is
\begin{align}
    A(t,i)^- :=& (\bar{O}(t-1), A(t,1),\ldots,A(t,i-1)) \\
    L(t,i)^- :=& (\bar{O}(t-1), A(t), L(t,1),\ldots,L(t,i-1)).
\end{align}
We let $F_L(t,i)$ be the set of individuals $i$ is in contact with at $t$, and we let $F_A(t,i)$ be the set of individuals upon whose history the experimenter decides the treatment assignment of individual $i$ at time $t$. We define the context functions $C^A(t,i)$ and $C^L(t,i)$ as
\begin{align}
    C^A(t,i) := c_{A(t,i)}(A(t,i)^-,F_A(t,i)) \\
    C^L(t,i) := c_{L(t,i)}(L(t,i)^-, F_L(t,i)).
\end{align}
When there is no ambiguity, we drop the $``L''$ superscript and we use $C(t,i)$ for $C^L(t,i)$.
We let
\begin{align}
    X(t,i) := (C^L(t,i), L(t,i)).
\end{align}

We denote $\widetilde{A}(k)$, $\widetilde{L}(k)$, $\widetilde{O}(k)$ and $\widetilde{X}(k)$ the $k$-th element in the sequences $(A(t,i))_{t,i}$, $(L(t,i))_{t,i}$, $(O(t,i))_{t,i}$, and $(X(t,i))_{t,i}$, respectively.

\subsection{Notation relative to probability distributions, their components, and the target parameters}

\begin{align}
    P^{T,N}_F :& \text{ probability distribution of the full data } (O^{T,N},U),\\
    P^{*,T,N}_F :& \text{ post-intervention distribution of the full data } (O^{*,T,N},U),\\
    P^{T,N} :& \text{ probability distribution of the observed data } O^{T,N},\\
    P^{T,N}_{g^*} :& \text{ G-computation formula}
\end{align}
We denote $\mathcal{M}^{T,N}_F$ the causal model, that is the set of possible full data distributions $P^{T,N}_F$, and $\mathcal{M}^{T,N}$ the statistical model that is the set of possible observed data distributions $P^{T,N}$. We now present notation for components of these distributions:
\begin{align}
    q :& \text{ conditional density of } L(t,i) \text{ given } L(t,i)^- \text{ (or given } C^L(t,i) \text{)} \\
    g = \prod_{t=1}^T \prod_{i=1}^N g_{t,i} :& \text{ treatment mechanism },\\
    g^* = \prod_{s=1}^\tau \prod_{j=1}^N g^*_{s,j} :& \text{ counterfactual treatment mechanism},\\
    h^A_{t,i}(q,g) \text{ and } h^L_{t,i}(q,g) :& \text{ marginal densities of } C^A(t,i) \text { and } C^L(t,i) \text { under } P^{T,N} = \prod_{t,i} g_{t,i} q,\\
    h^A_{t,i} \text{ and } h^L_{t,i} :& \text{ shorthand for } h^A_{t,i(q,g)} \text { and } h^L_{t,i}(q,g) \\
    h^{*,A}_{s,j} \text{ and } h^{*,L}_{s,j} :& \text{ shorthand for } h^A_{s,j}(q,g^*) \text { and } h^L_{s,j}(q,g^*),\\
    \bar{h}_{T,N}^A :=& (TN)^{-1} \sum_{t=1}^T \sum_{i=1}^N h_{t,i}^A, \\
    \bar{h}_{T,N}^L :=& (TN)^{-1} \sum_{t=1}^T \sum_{i=1}^N h_{t,i}^L,\\
    \omega_{s,j} :=& h^A_{s,j} / \bar{h}^A_{T,N} \text{ and } \eta_{s,j} := g^*_{s,j} / g_{s,j}.
\end{align}
When there is no ambiguity, we use $h_{t,i}$, $h^*_{s,j}$, $\bar{h}_{T,N}$ for $h^L_{t,i}$, $h^{*,L}_{s,j}$, $\bar{h}^L_{T,N}$. We denote $\widetilde{h}_k$ the $k$-th element of the column ordered sequence $(h_{t,i})_{t,i}$.
We now recall the definition of the causal parameter and the statistical target parameter:
\begin{align}
    \Psi^F_\tau(P^{T,N}_F) :=& E_{P^{*,T,N}_F}[ Y^*(\tau)] \text{ (causal parameter)},\\
    \Psi_\tau(P^{T,N}) :=& E_{P_{q,g^*}} [Y(\tau)] \text{ (statistical target parameter)}.
\end{align}
\section{Proofs of the structural results}

\subsection{Derivation of the efficient influence function}

The proof of theorem \ref{thm:EIF_representation} relies on the following lemma, which is a straightforward extension of lemma 1 in \cite{vdL2013}.

\begin{lemma}[Projection onto tangent space.]\label{lemma:projection_tangent_space}
The tangent space of the statistical model $\mathcal{M}$ at $P^{T,N}$ is given by
\begin{align}
    T(q) := \left\lbrace o^{T,N} \mapsto \sum_{t,i} s(l(t,i) \mid c^L(t,i)) : s:\mathcal{L} \times \mathcal{C} \to \mathbb{R},\ \forall c, \int s(l,c) q(l \mid c) dl = 0 \right\rbrace.
\end{align}
The projection of any function $o^{T,N} \mapsto D^*(o^{T,N})$ on $T(q)$ is 
\begin{align}
    \bar{D}: o^{T,N} \mapsto \frac{1}{T N} \sum_{t=1}^N \sum_{i=1}^N \frac{h_{t,i}(c^L)}{\bar{h}_{T,N}(c^L)} D_{t,i}(l \mid c),
\end{align}
with 
\begin{align}
    D_{t,i}(l,c) := E \left[ D^*(O^{T,N} \mid L(t,i) = l, C^L(t,i) = c^L \right] - E \left[ D^*(O^{T,N} \mid C^L(t,i) = c^L \right].
\end{align}
\end{lemma}

The proof is almost identical to that of lemma 1 in \cite{vdL2013}. We refer the interested reader to this work.

We now present the proof of theorem \ref{thm:EIF_representation}.

\begin{proof}[Proof of theorem \ref{thm:EIF_representation}] 
We start with the case $t = \tau$. We use the following classical strategy to find the canonical gradient: we first find a gradient of $\Psi$ at $q$ w.r.t. $T(q)$, and we then project it onto $T(q)$, which gives the canonical gradient.

\paragraph{Finding a gradient.} Consider a one-dimensional sub-model of $\mathcal{M}$ of the form \begin{align}
    \left\lbrace P^{T,N}_{\epsilon} : \epsilon \in [\pm \epsilon_{\max}], \frac{dP^{T,N}_{\epsilon}}{d \mu}(o^{T,N}) = \prod_{t,i} q_\epsilon(l(t,i) \mid c^L(t,i)) g^*(a(t,i) \mid c^A(t,i))  \right\rbrace,
\end{align}
such that $P^{T,N}_{\epsilon =0} = P^{T,N}$. We have that
\begin{align}
\Psi(P^{T,N}_\epsilon) = \Psi(q_\epsilon) = \int y \prod_{t,i} q_\epsilon(l(t,i) \mid c^L(t,i)) g^*(a(t,i) \mid c^A(t,i)) do^{T,N}. 
\end{align}
Therefore
\begin{align}
    \left.\frac{d \Psi(q_\epsilon)}{d \epsilon} \right\rvert_{\epsilon = 0} =& \int y \frac{d}{d \epsilon} \prod_{t,i} q_\epsilon \bigg|_{\epsilon = 0} \prod_{t,i} g^*_{t,i}   \\
    =& \int \prod_{t,i} q g_{t,i} \left\lbrace \prod_{t,i} \frac{g^*_{t,i}}{g_{t,i}} y \right\rbrace \frac{d}{d \epsilon} \log  \prod_{t,i} q_\epsilon  \bigg|_{\epsilon = 0} \\
    =& \int \prod_{t,i} q g_{t,i} \left\lbrace \prod_{t,i} \frac{g^*_{t,i}}{g_{t,i}} y - \Psi(q) \right\rbrace \frac{d}{d \epsilon} \log  \prod_{t,i} q_\epsilon  \bigg|_{\epsilon = 0} \\
    =& E_{q,g} \left[ \left( \prod_{t,i} \frac{g^*_{t,i}}{g_{t,i}} Y - \Psi(q) \right) \sum_{t,i} s(L(t,i), C^L(t,i)) \right],
\end{align}
where $s(l, c^L) := (d \log q_\epsilon / d \epsilon)|_{\epsilon=0}(l, c^L)$, and therefore $\sum_{t,i} s(l(t,i), c^L(t,i))$ is the score of the parametric submodel at $\epsilon = 0$.

Therefore
\begin{align}
    o^{T,N} \mapsto D^*(o^{T,N}) = \prod_{t,i} \frac{g^*_{t,i}}{g_{t,i}} y - \Psi(q)
\end{align}
is a gradient of $\Psi$ at $P^{T,N}$ w.r.t. $T(P)$.

\paragraph{Projecting $D^0$ onto the tangent space.} From lemma \ref{lemma:projection_tangent_space}, the projection of $D^0$ onto $T(q)$ is
\begin{align}
    D(q)(o^{T,N}) = \frac{1}{T N} \sum_{t,i} \bar{D}_{T,N}(q)(c^L(t,i), l(t,i)),
\end{align}
with 
\begin{align}
    \bar{D}_{T,N}(q)(c^L, l) = \sum_{s,j} \frac{h_{s,j}(c^L)}{\bar{h}_{T,N}(c^L)} &\left\lbrace  E_{q,g}\left[ Y g^* / g \mid L(t,i) =l, C^L(t,i) = c^L \right] \right.\\
    &\left.- E_{q,g}\left[ Y g^* / g \mid C^L(t,i) = c^L \right] \right\rbrace.
\end{align}

\paragraph{Second representation.} Suppose that assumption \ref{assumption:C_L_sufficient} holds. We have that
\begin{align}
    & E_{q,g} \left[ Y g^* / g \mid L(t,i) = l, C^L(t,i) = c^L \right] \\
    =& \frac{1}{h_{t,i}(c^L)} \int E_{q,g} \left[ Y g^* /g \mid L(t,i) = l, L(t,i)^- = l(t,i)^-\right]  \\
    & \qquad \qquad \times \Ind(c^L(l(t,i)^-) = c^L) \prod_{(s, j) < (t,i)} q(l(s,j) \mid l(s,j)^-) g_{s,j}(a(s,j) \mid a(s,j)^-) do^{T,N} \\
    =& \frac{1}{h_{t,i}(c^L)} \int E_{q,g^*} \left[ Y \mid L(t,i) = l, L(t,i)^- = l(t,i)^-\right]  \\
    & \qquad \qquad \times \Ind(c^L(l(t,i)^-) = c^L) \prod_{(s, j) < (t,i)} q(l(s,j) \mid l(s,j)^-) g^*_{s,j}(a(s,j) \mid a(s,j)^-) do^{T,N} \\
    = &  \frac{1}{h_{t,i}(c^L)} E_{q,g^*} \left[ Y \mid L(t,i) = l, c^L(L(t,i)^-) = c^L\right] \\
    &\times \int \Ind(c^L(l(t,i)^-) = c^L) \prod_{(s, j) < (t,i)} q(l(s,j) \mid l(s,j)^-) g^*_{s,j}(a(s,j) \mid a(s,j)^-) do^{T,N} \\
    =& \frac{h_{t,i}^*(c^L)}{h_{t,i}(c^L)} E_{q,g^*} \left[ Y \mid L(t,i) = l, C^L(t,i) = c^L \right].
\end{align}

Similarly,
\begin{align}
    E_{q,g} \left[Y g^* /g \mid C^L(t,i) = c^L \right] = \frac{h^*_{t,i}(c^L)}{h_{t,i}(c^L)} E_{q,g^*} \left[ Y \mid C^L(t,i) = c^L \right].
\end{align}
Replacing these expression in the expression of the canonical gradient gives the wished representation.

\paragraph{Third representation.} Under assumption \ref{assumption:C_L_decomposition}, the third representation follows immediately from the second one.
\end{proof}

\subsection{Proofs of the results on the remainder term}

\begin{proof}[Proof of theorem \ref{thm:first_second_order_result}]
Suppose $\bar{h}_{T,N} = \bar{h}_{0,T,N}$. We have that
\begin{align}
    &E_{P_0^{T,N}} \left[D(q)(O^{T,N})\right] \\
    =& E_{P_0^{T,N}} \left[ \frac{1}{T N} \sum_{t=1}^T \sum_{i=1}^N \bar{D}_{T,N}(C(t,i), L(t,i)) \right] \\
    =& \int \bar{h}_{0,T,N}(c) q_0(l\mid c) \bar{D}_{T,N}(c, l) dc dl \\
    =& \sum_{s=1}^\tau \sum_{j=1}^N \left\lbrace \int h^*_{s,j}(c) q_0(l \mid c) E_{q,g^*} \left[ Y \mid L(s,j) = l, C(s,j) = c\right] dl dc \right. \\
    &\qquad \qquad \left. - \int h^*_{s,j}(c) E_{q,g^*}\left[Y \mid C(s,j) = c \right] dc \right\rbrace
\end{align}
We have that 
\begin{align}
    &\int h^*_{s,j}(c) q_0(l \mid c) E_{q,g^*}\left[Y \mid L(s,j) = l, C(s,j) = c\right] \\
    =& \int q_0(l \mid c) E_{q,g^*} \left[Y \mid L(s,j) = l, L(s,j)^- = l(s,j)^- \right] \\
    &\qquad \qquad \times \Ind(c_{L(s,j)}(l(s,j)^-)=c) \prod_{(s',j') < (s,j)} q_{s',j'} g^*_{s,j} dl dc d(l(s,j)^-) \\
    =& \int E_{q,g^*} \left[Y \mid L(s,j) = l, L(s,j)^- = l(s,j)^- \right] q_{0,s,j} \prod_{(s',j') < (s,j)} q_{s',j'} g^*_{s,j} \\
    =& E_{q_{(s,j)^-}, q_{0,s,j} q_{(s,j)^+}} Y \label{eq:pf_first_remainder_res_eq1},
\end{align}
where the last line was obtained by using Fubini's theorem and integrating out the indicator.
\end{proof}
The same arguments show that
\begin{align}
    \int h_{s,j}^*(c) E_{q,g^*} \left[ Y \mid C(s,j) = c\right] = E_{q,g^*}[Y].
\end{align}
Therefore,
\begin{align}
    E_{P_0^{T,N}}\left[D(q)(O^{T,N}) \right] = \sum_{s=1}^\tau \sum_{j=1}^N E_{q_{(s,j)^-}, (q_0 - q)_{(s,j)}, q_{(s,j)}^+, g^*} Y.
\end{align}
Using the telescoping sum formula for the product difference $\prod_{s,j} q_{s,j} - \prod_{s,j} q_{0,s,j}$,
we have that 
\begin{align}
    \Psi(q) - \Psi(q_0) = \sum_{s,j} E_{q_{(s,j)^-}, q_{(s,j)}-q_{0,(s,j)}, q_{0,(s,j)}, g^*} Y.\label{eq:pf_first_remainder_res_eq2}
\end{align}
Therefore, putting \eqref{eq:pf_first_remainder_res_eq1} and \eqref{eq:pf_first_remainder_res_eq2} together gives the wished expression for $R(\bar{h}_{0,T,N},q)$. 

The derivation of the expression $R(\bar{h}_{0,T,N}, q) - R(\bar{h}_{T,N},q)$ is immediate.

\begin{proof}[Proof of theorem \ref{thm:second_second_order_result}]
We have that
\begin{align}
    &E_{P_0^{T,N}} \left[D(q)(O^{T,N}) \right] \\
    =& E_{P_0^{T,N}} \left[\frac{1}{T N} \sum_{t=1}^T \sum_{i=1}^N \bar{D}_{T,N}(q)(C^A(t,i), A(t,i), L(t,i)\right] \\
    =& \sum_{s=1}^\tau \sum_{j=1}^N \int \bar{h}^A_{0,T,N}(c^A) g_{s,j}(a \mid c^A) q_0(l \mid a, c^A) \omega_{s,j}(c^A) \eta_{s,j}(a \mid ^A) \\
    &\qquad \times \left\lbrace E_{q,g^*} \left[Y(j) \mid L(s,j) = l, A(s,j) = a, C^A(s,j) = c^A \right] \right. \\
    & \qquad \left.- E_{q,g^*} \left[Y(j) \mid A(s,j) = a, C^A(s,j) = c^A \right]  \right\rbrace dl da dc^A,
\end{align}
where we have used assumption \ref{assumption:indiv_outcomes_indep_other_trajs} in the last line above.

\paragraph{Case $\omega = \omega_0$.} We show that in this case, for any given $j$ and $s$, the second term of the $(s,j)$-th term in the sum above cancels out with the first term of the $(s-1,j)$-th term.

We start with rewriting the second term of the $(s,j)$-th term:
\begin{align}
    &\int \bar{h}^A_{0,T,N}(c^A) g_{s,j}(a \mid c^A) q_0(l \mid a, c^A) \omega_{0,s,j}(c^A) \eta_{s,j}(a \mid c^A) \\
    & \qquad \times E_{q,g^*} \left[Y(j) \mid A(s,j) = a, C^A(s,j) = c^A \right] dl da dc^A \\
    =& \int \bar{h}^{*,A}_{0,s,j}(c^A) g^*_{s,j}(a \mid c^A) E_{q,g^*} \left[Y(j) \mid A(s,j) = a, C^A(s,j) =  c^A \right] da dc^A \\
    =& \int \prod_{t=1}^{s-1} g_{t,j}^* q_{0,t,j} g^*_{s,j} E_{q,g^*} \left[Y(j) \mid A(s,j) = a, \bar{O}(s-1,j) = \bar{o}(s-1,j) \right] da d\bar{o}(s-1,j) \\
    =& E_{q_{0,1:s-1}, q_{s:\tau}, g^*} Y.
\end{align}
The third line above follows from assumption \ref{assumption:C_L_sufficient}.
We now show that the first term of the $(s-1,j)$-th term is equal to the above quantity:
\begin{align}
    &\int \bar{h}^A_{0,T,N}(c^A) g_{s,j}(a \mid c^A) q_0(l \mid a, c^A) \omega_{s,j}(c^A) \eta_{s,j}(a \mid ^A) \\
     &\qquad \times E_{q,g^*} \left[Y(j) \mid L(s-1,j) = l, A(s-1,j) = a, C^A(s-1,j) = c^A \right] dl da dc^A\\
    =& \int \bar{h}^{*,A}_{0,s,j}(c^A) g^*_{s,j}(a \mid c^A) q_0(l \mid a, c^A) \\
    & \qquad \times E_{q,g^*} \left[Y(j) \mid L(s-1,j) = l, A(s-1,j) = a, C^A(s-1,j) = c^A \right] dl da dc^A \\
    =& \int \prod_{t=1}^{s-2} g_{t,j}^* q_{0,t,j}g^*_{s-1,j} q_{0,s-1} E_{q,g^*} \left[Y(j) \mid \bar{O}(s-1,j) = \bar{o}(s-1,j) \right] d\bar{o}(s-1,j) \\
    =& E_{q_{0,1:s-1}, q_{s:\tau}, g^*} Y.
\end{align}
Thus, by telescoping, $E_{P_0^{T,N}}[D(q)(O^{T,N})] =\Psi(q) - \Psi(q_0)$, and therefore $R(\omega_0, q) = 0$. 

At a high level, the reason why this cross-terms cancellation happens is the first term of the $(s-1,j)$-th term is obtained by integration against $g^*_{s,j}$ of the second term of $(s,j)$-th term. Applying the operator $E_{P_0^{T,N}}$ boils down to successively, in the backwards direction, integrating with respect to the factors of $P_0^{T,N}$. The first step in this process applied to the second term of $(s,j)$ is to integrate w.r.t. $g^*_{s,j}$, which gives the first term of $(s-1,j)$. The subsequent steps being the same for both terms, the resulting quantities are the same. 

\paragraph{Case $q = q_0$.} In this case, it is immediate to observe that the cancellation happens within each term of the terms of the sum over $(s,j)$. Therefore, $R(\omega, q_0) = 0$.

Therefore, we can write $R(\omega, q) = R(\omega, q) - R(\omega_0,q)$, which makes appear the wished product of differences structure.
\end{proof}

\section{Results on empirical process induced by weakly dependent sequences}

Let $(X_n)_{n \geq 1}$ be a sequence of random variables taking values in a set $\mathcal{X}$, and let $\mathcal{F}$ be a class of functions with domain $\mathcal{X}$. In this section, we present a several novel results on empirical processes of the form
\begin{align}
    \left\lbrace M_n(f) : f \in \mathcal{F} \right\rbrace \qquad \text{where } \qquad M_n(f) := \frac{1}{n}\sum_{i=1}^n f(X_i) - E[f(X_i)].
\end{align}

We present three types of results: a maximal inequality over $\mathcal{F}$ (or over the intersection of $\mathcal{F}$ with a ball of controlled radius), an equicontinuity result, and an exponential risk bound for empircal risk mimizers over $\mathcal{F}$. The latter two are a consequence of the former.

We do not make independence nor stationarity assumptions on the sequence $(X_n)_{n \geq 1}$. Rather, we consider sequences $(X_n)_{n \geq 1}$ that satisfy only the following mixing conditions.

\begin{assumption}[$\alpha$-mixing]\label{assumption:exponential_alpha_mixing}
The uniform $\alpha$-mixing coefficients of the sequence $(X_i)_{i \geq 1}$ satisfy
\begin{align}
    \alpha(n) \leq \exp(-2 c n), \text{ for some } c > 0.
\end{align}
\end{assumption}

\begin{assumption}[$\rho$-mixing]\label{assumption:finite_rho_mixing}
The uniform $\rho$-mixing coefficients of the sequence $(X_i)_{n \geq 1}$ have finite sum, that is $\sum_{n \geq 1} \rho(n) < \infty$. We denote $\bm{\rho} := \sum_{n \geq 1} \rho(n)$.
\end{assumption}

We suppose that $\mathcal{X} \subset \mathbb{R}^d$ for some $d \geq 1$ and that, for every $i \geq 1$, the marginal distribution of $X_i$ admits a density w.r.t. the Lebesgue measure that we denote $h_i$. Supposing that assumption \ref{assumption:finite_rho_mixing} holds, we define the following mapping $\mathcal{F} \to \mathbb{R}$:
\begin{align}
    \sigma(f) := \sqrt{1 + 2 \bm{\rho}} \sup_{i \geq 1} \| f \|_{2,h_i}. 
\end{align}
It is straightforward to check that $\sigma$ is a norm. Our results apply to classes of functions that are bounded in supremum norm.

\begin{assumption}[Uniform boundedness]\label{assumption:max_ineq_unif_bound}
There exists $M \in (0, \infty)$ such that $\sup_{f \in \mathcal{F}} \|f\|_\infty \leq M$.
\end{assumption}

\subsection{A local maximal inequality}\label{subsection:local_max_ineq}

The result of this subsection is a \textit{local} maximal inequality in the sense that it bounds the supremum of $M_n(f)$ over a $\sigma$-ball included in $\mathcal{F}$. We state below the corresponding assumption.

\begin{assumption}[$\sigma$ norm boundedness]\label{assumption:max_ineq_sigma_bound} There exists $r > 0$ such that $\sup_{f \in \mathcal{F}} \sigma(f) \leq r$.
\end{assumption}

We can now state our result.

\begin{theorem}\label{thm:max_ineq_empirical_mixing_process}
Suppose that assumptions \ref{assumption:exponential_alpha_mixing}, \ref{assumption:finite_rho_mixing}, \ref{assumption:max_ineq_unif_bound} hold. Suppose that $r \geq 2 M n^{-1/2}$. Then, for any $r^ \in [M n^{-1/2}, r]$, it holds with probability at least $1 - 2e^{-x}$ that,
\begin{align}
    \sup_{f \in \mathcal{F}} M_n(f) \lesssim &r^- + \frac{\log n}{\sqrt{n}} \int_{r^-}^r \sqrt{\log (1 + N_{[\,]}(\epsilon, \mathcal{F}, \sigma))} d \epsilon + \frac{M (\log n)^2}{n} \log (1 + N_{[\,]}(r, \mathcal{F}, \sigma))\\
    &+ r \sqrt{\frac{(\log n) x}{n}} + \frac{M (\log n)^2}{n} x.
\end{align}
\end{theorem}

The proof relies on the following lemma, which is a corollary of lemma A.7 in \cite{vanHandel2010} and theorem 2 in \cite{merlevede2009}.

\begin{lemma}\label{lemma:cond_exp_max_finite_collection}
Suppose that $f_1, \ldots, f_N \in \mathcal{F}$, and that conditions on the preceding theorem hold. Then, for any event $A$ defined on the same probability space as $(X_i)_{n \geq 1}$,
\begin{align}
    E \left[ \max_{i \in [N]} M_n(f_i) \mid A \right] \lesssim \frac{1}{\sqrt{n}} \left (\max_{i \in N} \sigma(f_i) + \frac{M}{\sqrt{n}}\right) \sqrt{\log \left( 1 + \frac{N}{P[A]}\right)}  + \frac{M (\log n)^2}{n} \log \left( 1 + \frac{N}{P[A]}\right).
\end{align}
\end{lemma}

\begin{proof}[Proof of theorem \ref{thm:max_ineq_empirical_mixing_process}]
 The result will follow if we show that for $A := \left\lbrace \sup_{f \in \mathcal{F}} M_n(f) \geq \Psi(x) \right\rbrace$, 
 with
 \begin{align}
     \Psi(x) :=& r^- + \frac{\log n}{\sqrt{n}} \int_{r^-}^r \sqrt{\log (1 + N_{[\,]}(\epsilon, \mathcal{F}, \sigma))} d \epsilon + \frac{M (\log n)^2}{n} \log (1 + N_{[\,]}(r, \mathcal{F}, \sigma))\\
    &+ r \sqrt{\frac{(\log n) x}{n}} + \frac{M (\log n)^2}{n} x,
 \end{align}
 it holds that
 \begin{align}
     E \left[ \sup_{f \in \mathcal{F}} M_n(f) \mid A \right] \leq \Psi\left(\log \left( 1 + \frac{1}{P[A]} \right) \right)
 \end{align}

\paragraph{Setting up the notation.}
Let $\epsilon_j := r 2^{-j}$, and let $J \geq 1$ such that $\epsilon_J \leq r^- < \epsilon_{J-1}$. For every $j$, let 
\begin{align}
    \mathcal{B}_j := \{ ( \lambda^j_k, \upsilon^j_k) : k \in [N_j] \}
\end{align} 
be an $\epsilon_j$-bracketing of $\mathcal{F}$ in $\sigma$ norm. For every $f \in \mathcal{F}$ and every $j$, let $k(j,f)$ be such that 
\begin{align}
    \lambda^j_{k(j,f)} \leq f \leq \upsilon^j_{k(j,f)}.
\end{align}
For every $j$ and $f$, define the function $\Delta^j_f := \upsilon^j_{k(j,f)} - \lambda^j_{k(j,f)}$.
Let $a_j$ be a decreasing sequence of positive numbers, such that $a_{j-1}$ and $a_j$ are within constant factors of each other for every $j$.  We introduce, for every $f$, the function
\begin{align}
    \tau(f) : x \mapsto \left( \min \{ j \geq 0 : \Delta^j_f(x) > a_j \} \right) \wedge J.
\end{align}

\paragraph{Chaining decomposition.} We write $f$ as a telescoping sum using an adaptive chaining device. Adaptive chaining is a standard empirical process technique introduced by \cite{ossiander1987} for the analysis of empirical processes under bracketing entropy conditions, in which the depth of a chain is a function of the form of $\tau(f)$ that is chosen so as to control the supremum norm of the links of the chain. We have, in a pointwise sense, that, for every $f \in \mathcal{F}$, 
\begin{align}
    f =& \lambda^0_{k(0,f)} + \sum_{j=1}^{\tau(f)} \left(\lambda^j_{k(j,f)} - \lambda^{j-1}_{k(j-1,f)}\right)  + \left(f - \lambda^{\tau(f)}_{k(\tau(f), f)} \right)\\
    =& \lambda^0_{k(0,f)} + \sum_{j=1}^J \left( \lambda^j_{k(j,f)} - \lambda^{j-1}_{k(j-1,f)} \right) \Ind(\tau(f) < j) \\
    &+ \sum_{j=1}^J \left( f - \lambda^j_{k(j,f)} \right) \Ind(\tau(f) = j).
\end{align}
The first term represents the root of the chain. The second term is the sum across depth levels $j$ of the links of the chain. The third term is the tip of the chain.

\paragraph{Control of the tips.} We treat separately the case $j <J$ and the case $j  = J$.
\subparagraph{Case $j < J$.}
We will use the fact that, for $j < J$, we must have that if $\tau(f) =j$, then $\Delta^j_f > a_j$. From non-negativity of $f - \lambda^j_{k(j,f)}$,
\begin{align}
    M_n( (f - \lambda^j_{k(j,f)})\Ind(\tau(f) = j) )  \leq & E \left[ \frac{1}{n} \sum_{i=1}^n (f - \lambda^j_{k(j,f)})(X_i) \Ind(\tau(f)(X_i) = j) \right] \\
    \leq & E \left[ \frac{1}{n} \sum_{i=1}^n \Delta^j_f(X_i)  \Ind(\tau(f)(X_i) = j)  \right].
\end{align}
As $\Delta^j_f \Ind(\tau(f) = j) > a_j \Ind(\tau(f) = j)$, we have that $\Delta^j_f \Ind(\tau(f) =j) \leq a_j^{-1} (\Delta^j_f)^2 \Ind(\tau(f) = j)$, and therefore
\begin{align}
    E \left[ \frac{1}{n} \sum_{i=1}^n \left( \Delta^j_f \Ind(\tau(f) = j)\right)(X_i) \right] \leq & \frac{1}{a_j} \frac{1}{n} \sum_{i=1}^n E \left[ \left(\Delta^j_f(X_i) \right)^2 \right] \\
    = & \frac{1}{a_j} \frac{1}{n} \sum_{i=1}^n \| \Delta^j_f \|^2_{2,h_i} \\
    \leq & \frac{1}{a_j} \sigma^2(\Delta^j_f) \\
    \leq & \frac{\epsilon_j^2}{a_j}.
\end{align}
Therefore
\begin{align}
    E \left[ \sup_{f \in \mathcal{F}} M_n\left( (f - \lambda^j_{k(j,f)}) \Ind(\tau(f) = j) \right)  \right] \leq \frac{\epsilon_j^2}{a_j}.
\end{align}

\subparagraph{Case $j = J$.} We have that
\begin{align}
    M_n\left( (f - \lambda^j_{k(j,f)})\Ind(\tau(f) = J) \right) \leq & \frac{1}{n} \sum_{i=1}^n \| \Delta^j_f \|_{1,h_i} \\
    \leq & \frac{1}{n} \sum_{i=1}^n \| \Delta^j_f \|_{2,h_i} \\
    \leq & \sigma(\Delta^j_f) \\
    \leq & \epsilon_J.
\end{align}

\paragraph{Control of the links.} From the triange inequality
\begin{align}
    \sigma\left(\lambda^j_{k(j,f)} - \lambda^{j-1}_{k(j-1,f)} \right) \leq & \sigma\left(f - \lambda^j_{k(j,f)} \right) + \sigma\left(f- \lambda^{j-1}_{k(j-1,f)} \right) \\
    \leq & \sigma\left( \Delta^j_f\right) + \sigma\left( \Delta^{j-1}_f\right) \\
    \leq & \epsilon_j + \epsilon_{j-1} \\
    \lesssim & \epsilon_j.
\end{align}
Similarly
\begin{align}
    \left\lVert \left(\lambda^j_{k(j,f)} - \lambda^{j-1}_{k(j-1,f)} \right) \Ind(\tau(f) < j) \right\rVert_\infty \leq & \left\lVert\Delta^j_{k(j,f0} \Ind(\tau(f) < j) \right\rVert_\infty + \left\lVert \Delta^{j-1}_{k(j-1,f0} \Ind(\tau(f) < j) \right\rVert_\infty \\
    \leq & a_j + a_{j-1} \\
    \lesssim & a_j.
\end{align}
When $f$ varies over $\mathcal{F}$, $\lambda^j_{k(j,f)} - \lambda^{j-1}_{k(j-1,f)}$ varies over a collection of at most $\bar{N}_j := \prod_{l=0}^j N_l$ functions. From lemma \ref{lemma:cond_exp_max_finite_collection}, we thus have that
\begin{align}
     E \left[ M_n \left( \lambda^j_{k(j,f)} - \lambda^{j-1}_{k(j-1,f)} \right) \mid A \right] \lesssim & \frac{1}{\sqrt{n}} \left( \epsilon_j + \frac{M}{\sqrt{n}}\right) \sqrt{ \log \left( 1 + \frac{\bar{N}_j}{P[A]}\right)} + \frac{a_j (\log n)^2}{n} \log \left( 1 + \frac{\bar{N}_j}{P[A]}\right) \\
     \lesssim & \frac{1}{\sqrt{n}} \epsilon_j \sqrt{ \log \left( 1 + \frac{\bar{N}_j}{P[A]}\right)} + \frac{a_j (\log n)^2}{n} \log \left( 1 + \frac{\bar{N}_j}{P[A]}\right),
\end{align}
as $\epsilon_j \geq r^- \geq M n^{-1/2}$.

\paragraph{Control of the root.}
From the triangle inequality,
\begin{align}
    \sigma(\lambda^0_{k(j,f)}) \leq & \sigma( f - \lambda^0_{k(0,f)} ) + \sigma(f) \\
    \leq & 2 \epsilon_0 .
\end{align}
In addition, we have that $\| \lambda^0_{k(j,f)} \|_\infty \leq M$ (if not, we can always, without loss of generality, truncate it to $[-M,M]$ without altering its bracketing properties).

Therefore,
\begin{align}
    E \left[ \sup_{f \in \mathcal{F}} M_n(f) \mid A \right] \lesssim \frac{\epsilon_0}{ \sqrt{n}} \sqrt{\log \left( 1 +\frac{N_0}{P[A]}\right)} + \frac{M (\log n)^2}{n} \log \left( 1 +\frac{N_0}{P[A]}\right).
\end{align}

\paragraph{Adding up the bounds.} We obtain
\begin{align}
    E \left[ \sup_{f \in \mathcal{F}} M_n(f) \mid A \right] \lesssim & \frac{\epsilon_0}{\sqrt{n}} \sqrt{\log \left(1 + \frac{\bar{N}_0}{P[A}\right)} + \frac{M (\log n)^2}{n} \log \left( 1 + \frac{\bar{N}_0}{P[A]}\right) \\
    & + \sum_{j=1}^J \frac{\epsilon_j}{\sqrt{n}} \sqrt{\log\left( 1 + \frac{\bar{N}_j}{P[A]}\right)} + \frac{a_j (\log n)^2}{n} \log \left( 1 + \frac{\bar{N}_0}{P[A]}\right) \\
    &+ \sum_{j=0}^{J-1} \frac{\epsilon^2_j}{a_j} + \epsilon_J.
\end{align}
Set $a_j := \sqrt{n}  (\log n)^{-1} (\log (1 + \bar{N}_j / P[A]))^{-1/2}$. Then above bound then becomes
\begin{align}
    E \left[ \sup_{f \in \mathcal{F}} M_n(f) \mid A \right] \lesssim & \frac{\epsilon_J}{\sqrt{n}} + \frac{\log n}{\sqrt{n}} \sum_{j=0}^J \epsilon_j \sqrt{\log \left(1 + \frac{\bar{N}_j}{P[A]}\right)} + \frac{(\log n)^2}{n} M \log \left( 1 + \frac{N_0}{P[A]} \right).
\end{align}
Using the same arguments as at the end of the proof of theorem 5 in \cite{bibaut20-GPE}, we have that
\begin{align}
    \sum_{j=0}^J \epsilon_j \sqrt{\log \left(  1 + \frac{\bar{N}_j}{P[A]} \right)} \lesssim \int_{\epsilon_J}^{\epsilon_0} \sqrt{\log \left( 1 + N_{[\,]}(\epsilon, \mathcal{F}, \sigma) \right)} d \epsilon + \epsilon_0 \sqrt{\log \left(1 + \frac{1}{P[A]}\right)}.
\end{align}
The above and the fact that $\log (1 + N_0 / P[A]) \leq \log (1 + N_0) + \log (1 + 1/P[A])$ yield the wished claim.
\end{proof}

\subsection{Equicontinuity}\label{subsection:equicont_weakly_dep_EP}

\begin{theorem}\label{thm:equicont_weakly_dep_EP}
Consider a class of functions $\mathcal{F}$ and a sequence $(f_n)$ of elements of $\mathcal{F}$.

Suppose that conditions \ref{assumption:exponential_alpha_mixing}, \ref{assumption:finite_rho_mixing} and \ref{assumption:max_ineq_unif_bound} hold.

Suppose that there exists a deterministic sequence of positive numbers $(a_n)$ such that
\begin{align}
    a_n^{-2} (\log n)^2 / \sqrt{n} = o(1) \qquad \text{and} \qquad a_n^\nu \log n = o(1) \text{ for every } \nu > 0, \label{eq:rate_a_n}
\end{align}
and, for every $\epsilon > 0$, there exists $C_n(\epsilon) > 0$ such that
\begin{align}
    P \left[ \forall n \geq 1, \|f_n\|_\infty \leq C(\epsilon) a_n \right] \geq 1 - \epsilon. \label{eq:unif_bound_sup_norm_f_n}.
\end{align}
Suppose further that there
\begin{align}
    \log N_{[\,]}(u, \mathcal{F}, \sigma) \lesssim u^p \text{ for some } p > 0. \label{eq:Donsker_cond}
\end{align}
Then
\begin{align}
    \sqrt{n} M_n(f) = o(1) \text{ a.s.}
\end{align}
\end{theorem}

\begin{proof}[Proof of theorem \ref{thm:equicont_weakly_dep_EP}] Let $\epsilon > 0$, and let $C(\epsilon)$ as in the conditions of the theorem, and introduce the event 
\begin{align}
    \mathcal{E}_1(\epsilon) := \left\lbrace \forall n \geq 1, \|f_n\|_\infty \leq C(\epsilon) a_n \right\rbrace.
\end{align}
Observe that under $\mathcal{E}_1(\epsilon)$, for every $n \geq 1$,
\begin{align}
    \sqrt{n} M_n(f_n) \leq \sup \left\lbrace \sqrt{n} M_n(f) : f \in \mathcal{F}, \|f_n\|_\infty \leq C(\epsilon) a_n \right\rbrace.
\end{align}
We now bound with high probability the supremum in the right-hand side above, for every $n$. Let $x_n := \log (\epsilon / (n(n+1)))$. From theorem \ref{thm:max_ineq_empirical_mixing_process}, with probability at least $1 - \epsilon / (n(n+1))$,
\begin{align}
    \sup \left\lbrace \sqrt{n} M_n(f) : f \in \mathcal{F}, \|f_n\|_\infty \leq C(\epsilon) a_n \right\rbrace \lesssim \psi_n(\epsilon, a_n, x_n),
\end{align}
with 
\begin{align}
    \psi_n(\epsilon, a_n, x_n) :=& C(\epsilon) a_n + \log n \int_{\frac{C(\epsilon) a_n}{ \sqrt{n}}}^{C(\epsilon) a_n} u^{-p/2} du + (C(\epsilon) a_n)^{-p} \frac{(\log n)^2}{\sqrt{n}}\\
    &+ C(\epsilon) a_n \sqrt{\log (n(n+1) / \epsilon) \log n} + \frac{C(\epsilon) a_n}{\sqrt{n}} \log(n(n+1) / \epsilon) (\log n)^2.
\end{align}

Therefore, from a union bound,
\begin{align}
    P\left[ \exists n \geq 1 \sqrt{n} M_n(f_n) \geq \psi_n(\epsilon, a_n, x_n) \right] \leq & P[\mathcal{E}_1(\epsilon)^c] + \sum_{n=1}^\infty \frac{\epsilon}{n(n+1)} \\
    \leq & 2 \epsilon.
\end{align}
Since, for every $\epsilon > 0$, condition \ref{eq:rate_a_n} implies that $\psi_n(\epsilon, a_n, x_n) = o(1)$, the above implies that,
\begin{align}
    \forall \epsilon > 0, \ P\left[ \lim_{n \to \infty} \sqrt{n} M_n(f) = 0 \right] \leq 2 \epsilon,
\end{align}
which,  by letting $\epsilon \to 0$, implies the wished claim.
\end{proof}

\paragraph{Discussion of the supremum norm convergence requirement.} As we pointed out in the main text, it might appear surprising at first that even though our Donsker condition involves the entropy w.r.t. the $\sigma$ norm, which is an $L_2$ norm, we do need convergence in sup norm of $(f_n)_{n \geq 1}$.

The reason why this is the case can be understood from the expression and conditions of our maximal inequality for weakly dependent empirical processes, theorem \ref{thm:max_ineq_empirical_mixing_process} from the previous subsection. Recall that this result tells us that, under mixing conditions, given a class of functions $\mathcal{F}$ such that $\sup_{f \in \mathcal{F}} \sigma(f) \leq r$, and $\sup_{f \in \mathcal{F}} \|f \|_\infty \leq M$, for some $M, r > 0$, then for any $r^- \geq M / \sqrt{n}$, it holds with probability at least $1-2e^{-x}$ that
\begin{align}
    \sup_{f \in \mathcal{F}} M_n(f) \leq r^- + \frac{\log n}{\sqrt{n}} \int_{r^-}^r \sqrt{\log (1 + N_{[\,]}(\epsilon, \mathcal{F}, \sigma)} d \epsilon + r \log n \sqrt{\frac{x}{n}} + \frac{(\log n)^2 M x}{n}, \label{eq:max_ineq_restated}
\end{align}
where $M_n(f) = n^{-1} \sum_{i=1}^n f(X_i) - E f(X_i)$. Suppose that we want to prove an asymptotic equicontinuity result of the form $M_n(f_n) = o_P(n^{-1/2})$ for a certain sequence $(f_n)_{n \geq 1}$, while only assuming that $\sigma(f_n) = o_P(1)$ and not making any assumptions on $(\|f_n\|_\infty)_{n \geq 1}$. Then, as $n\to \infty$, we can bound $M_n(f_n)$ with high probability by the supremum of $M_n(f)$ over subsets of $\mathcal{F}$ with $\sigma$ radius arbitrarily close to zero. This allows us to bound $M_n(f_n)$ with high probability by the right-hand side above with the upper bound $r$ of the entropy integral arbitrarily close to zero. Unfortunately, letting the upper bound of the entropy integral converge to zero isn't enough to make the expression converge to zero faster than $n^{-1/2}$. Indeed, since the term $r^-$ needs to be at least as large as $M / \sqrt{n}$, with $M$ an upper bound on the $\|\cdot\|_\infty$ radius of the class over which we take the supremum, we need to be able to let $M$ get arbitrarily close to zero with high probability to make the RHS of \eqref{eq:max_ineq_restated} go to zero faster than $n^{-1/2}$. This explains why, given our maximal inequality \eqref{eq:max_ineq_restated}, we need to control $(\|f_n\|_\infty)_{n \geq 1}$. 

That being said, one might still wonder why in our maximal inequality the lower bound $r^-$ of the entropy integral needs to be larger than $M / \sqrt{n}$, thus making us pay an approximation error price $r^- \geq M / \sqrt{n}$. The reason is that we obtain our result by applying a chaining device to the following deviation bound \cite{merlevede2009} for fixed $f$. 
Under exponential $\alpha$-mixing (condition \ref{assumption:exponential_alpha_mixing}) and finiteness of the sum of $\rho$-mixing coefficients, if $\| f \|_\infty \leq M$, their result gives that
\begin{align}
    P \left[ \sqrt{n} M_n(f) \geq x  \right] \lesssim \exp \left( - \frac{x^2}{\sigma(f)^2+ \frac{M^2}{n} + \frac{M x (\log n)^2}{\sqrt{n}} }\right), \label{eq:Merlevede_et_al_bound}
\end{align}
Compare this with the usual Bernstein inequality i.i.d. random variables:
\begin{align}
     P \left[ \sqrt{n} M_n(f) \geq x \right] \lesssim \exp\left(- \frac{x^2}{\|f\|_{2,P}^2 + \frac{M x}{\sqrt{n}} } \right) 
\end{align}
While the i.i.d. Bernstein inequality implies that $\sqrt{n}M_n(f)$ scales as the $L_2$ norm $\|f\|_{P,2}$ (as long as the ratio $\|f\|_{2,P} / \|f\|_\infty \gtrsim 1/\sqrt{n}$), the concentration bound for mixing sequences implies that it scales as $\sigma(f) + \|f\|_\infty / \sqrt{n}$. In the chaining argument, we consider $\epsilon_j$-bracketings in $\|\cdot\|_{2,P}$ norm in the i.i.d. case, and in $\sigma$ norm in the weakly dependent case, with $\epsilon_j = r 2^{-j}$, for increasingly large $j$, where $j$ has the interpretation of the depth of the chains.


Let us first discuss chaining in the i.i.d. case, and in the case where we want to obtain a bound on $E_P \sqrt{n} \sup_{f \in \mathcal{F}} M_n(f)$ (as opposed to obtaining a high probability bound on $\sup_{f \in \mathcal{F}} M_n(f)$, which is slightly more technical). Denote $\{ [\lambda_{j,k}, \upsilon_{j,k} ] : k \in [N_j] \}$ the $\epsilon_j$-bracketing of $\mathcal{F}$ used in the chaining device.  The contribution of depth $j$ to the final bound on $E_P\sqrt{n} \sup_{f \in \mathcal{F}} M_n(f)$ is a supremum over links $\lambda_{j, k} - \lambda_{j-1,k}$ between depths $j$ and $j-1$, which can essentially be bounded, using Bernstein's inequality, by
\begin{align}
 \sup_{k \in [N_j], k \in [N_{j-1}]} \|\lambda_{j, k} - \lambda_{j-1,k}\|_{P,2} \sqrt{\log N_{[\,]} (\epsilon_j,\mathcal{F},\|\cdot\|_{P,2})} \lesssim \epsilon_j \sqrt{\log N_{[\,]} (\epsilon_j,\mathcal{F},\|\cdot\|_{P,2})}.
\end{align}
(To be rigorous, the bound obtained from Bernstein's inequality has another term, but in adaptive chaining, we choose the maximal depth of the chains so that this term is no larger than the first one above).
By contrast, in the weakly dependent case, the concentration bound \eqref{eq:Merlevede_et_al_bound} gives that the corresponding contribution is bounded by
\begin{align}
 &\left( \sup_{k \in [N_j], k \in [N_{j-1}]} \sigma(\lambda_{j, k} - \lambda_{j-1,k}) + \frac{M}{\sqrt{n}} \right) \sqrt{\log N_{[\,]} (\epsilon_j,\mathcal{F},\sigma)} \\
 \lesssim & \left(\epsilon_j  + \frac{M}{\sqrt{n}} \right) \sqrt{\log N_{[\,]} (\epsilon_j,\mathcal{F},\sigma)}.
\end{align}
A consequence of this is that when the depth $j$ of the chain is such that $\epsilon_j \leq M / \sqrt{n}$, then the term $M/\sqrt{n}$ becomes the main scaling factor. It can be checked that as result of this, the sum of these bounds diverges as $j \to \infty$. This is why in our chaining decomposition, we impose that our chains must have depth no larger that $J$ such that $\epsilon_J \geq M / \sqrt{n}$. This gives us a bound involving an entropy integral with lower bound $\epsilon_J$ and an approximation error $\sqrt{n}\epsilon_J$.


\subsection{Exponential deviation bound for empirical risk minimizers}

Let $\ell$ be a functional defined on $\mathbb{R}^{\mathcal{X}}$, the space of functions $\mathcal{X} \to \mathbb{R}$, such that, for every $\mathbb{R}^{\mathcal{X}}$, $\ell(f)$ is a function $\mathcal{X} \to \mathbb{R}$. We call $\ell$ a loss function. For every $f : \mathcal{X} \to \mathbb{R}$, we define the population risk and the empirical risk as
\begin{align}
    R_n(f) := \frac{1}{n} \sum_{i=1}^n E[\ell(f)(X_i)] \qquad \text{and} \qquad
    \widehat{R}_n(f) := \frac{1}{n} \sum_{i=1}^n \ell(f)(X_i).
\end{align}
Let $f_n$ be an empirical risk minimizer over $\mathcal{F}$, that is an element of $\mathcal{F}$ such that 
\begin{align}
    \widehat{R}_n(f_n) = \inf_{f \in \mathcal{F}} \widehat{R}_n(f).
\end{align}
and $f^* \in \mathcal{F}$ be a minimizer of the population risk, that is a function such that $R_n(f^*) = \inf_{f \in \mathcal{F}} R_n(f)$. In this section, we give exponential bounds on the excess population risk $R_n(f_n) - R_n(f^*)$, and on the norm $\sigma(f - f^*)$. We rely on the following assumptions.

\begin{assumption}[Variance bound]\label{assumption:var_bound} For every $f \in \mathcal{F}$, $\sigma^2(\ell(f) - \ell(f^*)) \lesssim (R_n(f) - R_n(f^*))$.
\end{assumption}
Assumption \ref{assumption:var_bound} can be checked in common settings, such as in non-parametric regression settings when $\ell$ is the square loss and the dependent variable has bounded range. 

\begin{assumption}[A power of $\sigma$ dominates $\|\cdot\|_\infty$ over $\ell(\mathcal{F})$]\label{assumption:sigma_alpha_dominates_sup_norm} There exists $\alpha \in (0,1)$ such that, for all $f \in \mathcal{F}$, $\| \ell(f) - \ell(f^*) \|_\infty \lesssim (\sigma(\ell(f) - \ell(f^*)))^\alpha$.
\end{assumption}

\begin{assumption}[Excess risk dominates norm of difference]\label{assumption:excess_risk_dominates_norm_of_diff}
Suppose that $\sigma^2(f- f^*) \lesssim R_n(f) - R_n(f^*)$ for every $f \in \mathcal{F}$.
\end{assumption}

Assumption \ref{assumption:sigma_alpha_dominates_sup_norm} holds for instance if all functions in $\ell(\mathcal{F})$ are all $L$-Lipschitz for the same $L$, as formally presented in lemma \ref{lemma:sup_norm_bound_Lipschitz}.

\begin{assumption}[Entropy]\label{assumption:Donsker_loss_class}
There exists $p \in (0,2)$ such that
\begin{align}
    \log N_{[\,]}(\epsilon, \ell(\mathcal{F}), \sigma) \lesssim \epsilon^{-p}.
\end{align}
\end{assumption}

\begin{theorem}[Exponential deviation bound for ERM]\label{thm:exp_bound_for_ERM}
Suppose that $\mathcal{F}$ is a convex set, and that $\ell$ is convex on $\mathcal{F}$. Suppose that assumptions \ref{assumption:exponential_alpha_mixing}, \ref{assumption:finite_rho_mixing}, \ref{assumption:var_bound}, \ref{assumption:sigma_alpha_dominates_sup_norm}, \ref{assumption:excess_risk_dominates_norm_of_diff} and \ref{assumption:Donsker_loss_class} hold. Let
\begin{align}
    \phi_n : r \mapsto \frac{r^\alpha}{\sqrt{n}} + \frac{\log n}{\sqrt{n}} r^{1-p/2} + \frac{(\log n)^2}{n} r^{\alpha - p},
\end{align}
and let $r_n > 0$ such that $r_n^2/3 = \phi_n(r_n)$ (there exists such an $r_n$ from lemma \ref{lemma:sublinear_functions} applied to $3 \phi_n$). Let $r > 0$ such that
\begin{align}
    r \geq \max \left\lbrace n^{-\frac{1}{2(1-\alpha)}}, r_n, \sqrt{3} \log n \sqrt{\frac{x}{n}}, (\log n)^{\frac{2}{2-\alpha}} \left( \frac{3x}{n} \right)^{\frac{1}{2-\alpha}}\right\rbrace.
\end{align}
Then, with probability at least $1- 2 e^{-x}$, $R_n(f_n) - R_n(f^*) \lesssim r^2$ and $\sigma(f_n -f^*) \lesssim r$.
\end{theorem}

The proof of \ref{thm:exp_bound_for_ERM} is a relatively straightforward adaptation of the proof of lemma 13 in \cite{bartlett_jordan_mcauliffe2006}. It relies on the following two intermediate lemmas

\begin{lemma}\label{lemma:sublinear_functions}
Let $\phi:(0,\infty) \to \mathbb{R}_+$ such that $r \mapsto \phi(r) / r$ is strictly decreasing on $(0,\infty)$ and $\lim_{r \to 0^+} \phi(r) / r > 1$. Then, there exists a unique $r_* \in (0,\infty)$ such that $r_*^2 = \phi(r_*)$, and for any $r \in (0,\infty)$, $r^2 \geq \phi(r)$ if and only if $r \geq r^*$.
\end{lemma}

\begin{lemma}\label{lemma:intermediate_lemma_ERM}
Suppose that the assumptions of theorem \ref{thm:exp_bound_for_ERM} hold and let $r_n$ and $r$ be as defined in theorem \ref{thm:exp_bound_for_ERM}. Then, there exists a constant $C> 0$ such that, for any $x > 0$,
\begin{align}
    P \left[ \sup \left\lbrace M_n(\ell(f) - \ell(f^*)) : f \in \mathcal{F}, R_n(f) - R_n(f^*) \leq r^2  \right\rbrace \geq C r^2 \right] \leq 2 e^{-x}.
\end{align}
\end{lemma}

\begin{proof}[Proof of lemma \ref{lemma:sublinear_functions}] The claim follows directly from the fact that, since both $r \mapsto \phi(r) / r$ and $r \mapsto 1/r$ are strictly decreasing on $(0,\infty)$, $r \mapsto \phi(r) / r^2$ is also strictly decreasing on $(0,\infty)$.
\end{proof}

\begin{proof}[Proof of lemma \ref{lemma:intermediate_lemma_ERM}]
From assumption \ref{assumption:var_bound}, 
\begin{align}
    & P \left[\sup \left\lbrace M_n(\ell(f) - \ell(f^*)) : f \in \mathcal{F}, R_n(f) - R_n(f^*) \leq r^2  \right\rbrace \gtrsim r^2 \right] \\
     \leq & P \left[ \sup \left\lbrace M_n(\ell(f) - \ell(f^*)) : f \in \mathcal{F}, \sigma(\ell(f) - \ell(f^*)) \lesssim r  \right\rbrace \gtrsim r^2 \right]. \label{eq:intermediate_lemma_ERM_eq1}
\end{align}

Under assumption \ref{assumption:sigma_alpha_dominates_sup_norm}, $\| M_n(\ell(f) - \ell(f^*)) \|_\infty \lesssim r^\alpha$ for any $f \in \mathcal{F}$ such that $\sigma(\ell(f) - \ell(f^*)) \lesssim r$. Therefore, since $r > n^{-1/(2(1-\alpha))}$, implies $r > r^\alpha n^{-1/2}$, applying theorem \ref{thm:max_ineq_empirical_mixing_process} with $r^- = r^\alpha / \sqrt{n}$, we have that
\begin{align}
    P \left[ \sup \left\lbrace M_n(\ell(f) - \ell(f^*)) : f \in \mathcal{F}, \sigma(\ell(f) - \ell(f^*)) \lesssim r  \right\rbrace \gtrsim \psi_n(r,x) \right] \leq 2 e^{-x},
\end{align}
with 
\begin{align}
    \psi_n(r,x) :=& \frac{r^\alpha}{\sqrt{n}} + \frac{\log n}{\sqrt{n}} \int_{\frac{r^\alpha}{\sqrt{n}}}^r u^{-p/2} du + \frac{(\log n)^2}{n} r^{\alpha - p} \\
    &+ r \log n \sqrt{\frac{x}{n}} + r^\alpha (\log n)^\alpha \frac{x}{n}.
\end{align}
Observe that 
\begin{align}
    \psi_n(r,x) \leq \phi_n(r) + r \log n \sqrt{\frac{x}{n}} + r^\alpha (\log n)^\alpha \frac{x}{n}.\label{eq:intermediate_lemma_ERM_eq2}
\end{align}
From the definition of $r$ in the statement of theorem \ref{thm:max_ineq_empirical_mixing_process}, we have $r \geq r_n$, which from lemma \ref{lemma:sublinear_functions} implies that $r^2 /3 \geq \phi_n(r)$. We also have $r^2/3 \geq r \log n \sqrt{x/n}$, and $r^2 / 3 \geq r^\alpha (\log n)^2 x / n$. Therefore, $r^2 \geq \psi_n(r,x)$. Therefore, from \eqref{eq:intermediate_lemma_ERM_eq1} and \eqref{eq:intermediate_lemma_ERM_eq2}, we have 
\begin{align}
    &P \left[ \sup \left\lbrace M_n(\ell(f) - \ell(f^*)) : f \in \mathcal{F}, \sigma(\ell(f) - \ell(f^*)) \lesssim r  \right\rbrace \gtrsim r^2 \right] \\
    \leq & P \left[ \sup \left\lbrace M_n(\ell(f) - \ell(f^*)) : f \in \mathcal{F}, \sigma(\ell(f) - \ell(f^*)) \lesssim r  \right\rbrace \gtrsim \psi_n(r,x) \right] \\
    \leq & 2 e^{-x}.
\end{align}
\end{proof}

\begin{proof}[Proof of theorem \ref{thm:exp_bound_for_ERM}] From the convexity of $\mathcal{F}$ and the convexity of $\ell$ on $\mathcal{F}$, the following assertion holds for every $r \geq 0$:
\begin{align}
    \exists f \in \mathcal{F},\ \widehat{R}_n(f) - \widehat{R}_n(f^*) \leq 0 \text{ and } R_n(f) - R_n(f^*) \geq r^2
\end{align}
implies that
\begin{align}
    \exists f \in \mathcal{F},\ \widehat{R}_n(f) - \widehat{R}_n(f^*) \leq 0 \text{ and } R_n(f) - R_n(f^*) = r^2.
\end{align}
Using this fact, and the fact that by definition of $f_n$, $\widehat{R}_n(f_n) - \widehat{R}_n(f^*) \leq 0$, we have
\begin{align}
    & P \left[ \widehat{R}_n(f_n) - \widehat{R}_n(f^*) \geq r^2 \right] \\
    \leq & P \left[ \exists f \in \mathcal{F}, \widehat{R}_n(f) - \widehat{R}_n(f^*) \leq 0 \text{ and } R_n(f) - R_n(f^*) \geq r^2 \right] \\
    \leq &  P \left[ \exists f \in \mathcal{F}, \widehat{R}_n(f) - \widehat{R}_n(f^*) \leq 0 \text{ and } R_n(f) - R_n(f^*) = r^2 \right] \\
    \leq & P \left[ \sup \left\lbrace M_n(\ell(f) - \ell(f^*)) : f \in \mathcal{F},\ \sigma(\ell(f) - \ell(f^*)) \lesssim r \right\rbrace \geq r^2 \right] \\
    \leq & 2 e^{-x}
\end{align}
from lemma \ref{lemma:intermediate_lemma_ERM}
\end{proof}

\section{Proofs for the analysis of the TMLE}

\subsection{Proof of lemma \ref{lemma:stablization_variance_EIF} on the stabilization of the variance of the EIF}

\begin{proof}[Proof of lemma \ref{lemma:stablization_variance_EIF}]
It is immediate to check that $\bar{D}_{T,N}(q_0)$ and $\bar{D}_{0,\infty,N}$ are centered, and therefore
\begin{align}
    \Var \left(\bar{D}_{T,N}(q_0)(L(t,i), C(t,i)) \right) =& \| \bar{D}_{T,N}(q_0) \|_{2, q_0, h_{0,t,i}}^2, \\
    \Var \left( \bar{D}_{0,\infty,N}(L_\infty, C_\infty) \right) =& \|\bar{D}_{0,\infty,N}\|_{2, q_0, h_{0,\infty,N}}^2.
\end{align}
We have that
\begin{align}
    & \| \bar{D}_{T,N}(q_0) \|_{2, q_0, h_{0,t,i}} - \|\bar{D}_{0,\infty,N}\|_{2, q_0, h_{0,\infty,N}} \\\
    =&  \| \bar{D}_{T,N}(q_0) \|_{2, q_0, h_{0,t,i}} -  \| \bar{D}_{0,\infty,N} \|_{2, q_0, h_{0,t,i}} \\
    &+ \| \bar{D}_{0,\infty,N} \|_{2, q_0, h_{0,t,i}} - \|\bar{D}_{0,\infty,N}\|_{2, q_0, h_{0,\infty,N}}.
\end{align}
We first start with the first term. We have that 
\begin{align}
    &\left\lvert \| \bar{D}_{T,N}(q_0) \|_{2, q_0, h_{0,t,i}} -  \| \bar{D}_{0,\infty,N} \|_{2, q_0, h_{0,t,i}} \right\rvert \\
    \leq &  \| \bar{D}_{T,N}(q_0) - \bar{D}_{0,\infty,N} \|_{2, q_0, h_{0,t,i}} \\
    \leq & \sum_{s=1}^\tau \sum_{j=1}^N \left\lVert \frac{1}{\bar{h}_{0,T,N}} - \frac{1}{h_{0,\infty,N}} \right\rVert_{2,h_{0,t,i}} \| h^*_{0,s,j} \widetilde{D}_{s,j,N}(q_0) \|_\infty \\
    \leq & 2 B \bm{\varphi} \tau \left\lVert \frac{1}{\bar{h}_{0,T,N}} - \frac{1}{h_{0,\infty,N}} \right\rVert_{2,h_{0,t,i}} \\
    =& o(1).
\end{align}
The third line above follows from the triangle inequality. The fourth line above is a consequence of lemma \ref{lemma:bound_tildeD_phi_mixing}. The fifth line follows from assumption \ref{assumption:cvgence_inverse_marginals_bounded_marginal_ratios}.

We now turn to the second term. We have that
\begin{align}
     &\| \bar{D}_{0,\infty,N} \|^2_{2, q_0, h_{0,t,i}} - \|\bar{D}_{0,\infty,N}\|^2_{2, q_0, h_{0,\infty,N}} \\
     =& E \left[  \bar{D}^2_{0,\infty,N}(L(t,i), C(t,i)) \right] - E \left[  \bar{D}^2_{0,\infty,N}(L_\infty, C_\infty) \right] \\
     & \to \infty 0,
\end{align}
since, from assumption \ref{assumption:ergodicity}, $((L(t,i), C(t,i)) \to (L_\infty, C_\infty))$, and $\bar{D}_{0, \infty, N}$ is a bounded continuous function.
\end{proof}

\subsection{Proof of the weak convergence of the martingale term}

\begin{proof}[Proof of \ref{thm:weak_convergence_M1xTN}]
 Order the couples $(t,i)$ as $(1,1),\ldots,(1,N),\ldots,(T,1),\ldots,(T,N)$, and let $(t(k), i(k))$ be the $k$-th couple in this ordering. Let 
 \begin{align}
     Z_{k,TN} := \frac{\bar{D}_{TN}(q_0)(X(t(k), i(k))}{\sigma_{0,\infty,N} \sqrt{TN}}.
 \end{align}
 Observe that under assumptions \ref{assumption:phi_mixing}, \ref{assumption:marginal_density_ratios_bound}, from lemma \ref{lemma:bound_barD_phi_mixing}, $\|\bar{D}_{TN}(q_0)\|_\infty \leq C$, for some $C < \infty$ that does not depend on $N$.
 
 Therefore, assumption \eqref{eq:cond1_McLeish} in theorem \ref{thm:mcleish_FCLT} is trivially checked. Let us now turn to assumption \eqref{eq:cond2_McLeish}. First, observe that, as $E[Z_{k,TN}^2] = \Var_{q_0, h_{0,i(k), t(k)}}(\bar{D}_{TN}(q_0)(X(t(k), i(k)) / (TN \sigma_{0,\infty,N}^2)$, we have that $NT E[Z_{k,TN}^2] \to 1$ as $k \to \infty$, and therefore, by Cesaro's lemma for deterministic sequences of real valued numbers,
 \begin{align}
     \sum_{k=1}^{\left\lfloor x T N \right\rfloor} E[Z_{k,TN}^2] \to x.
 \end{align}
 We now need to show that $V_{TN}(x) := \sum_{k=1}^{\left\lfloor x T N \right\rfloor} Z_{k,TN}^2$ converges in probability to its mean as $T \to \infty$. We proceed by taking the variance of $V_{TN}(x)$ and showing that it converges to zero, which, by Chebyshev's inequality will give us the wished result. We have
 \begin{align}
     \Var(V_{TN}(x)) =& \sum_{k=1}^{\left\lfloor x N T \right\rfloor} \Var(Z^2_{k,TN}) + \sum_{1 \leq k_1 < k_2 \leq \left\lfloor x N T \right\rfloor} \Cov(Z^2_{k_1,TN}, Z^2_{k_2,TN}) \\
     \leq & x N T \|Z_{k,TN} \|^4_\infty + \sum_{k_1=1}^{ \left\lfloor x N T \right\rfloor} \sum_{k_2 = k_1 + 1}^{ \left\lfloor x N T \right\rfloor} \|Z_{k_1, TN}\|_\infty^4 \alpha \left( X(t(k_1), i(k_1)), X(t(k_2), i(k_2)\right) \\
     \leq & x N T \frac{C^4}{ (TN)^2  \sigma_{0, \infty, N}^4} (1 + o(NT)) \\
     =& o(1),
 \end{align}
where the third line follows from assumption 5. By Chebyshev's inequality, we thus obtain that
\begin{align}
     \sum_{k=1}^{\left\lfloor x N T \right\rfloor} Z_{k,TN}^2 - E \left[Z_{k,TN}^2 \right] \to 0, 
\end{align}
which implies that \eqref{eq:cond2_McLeish}. This concludes the proof.
\end{proof}

\section{Proof of the exponential deviation bound for nuisance estimators}

\begin{proof}[Proof of theorem \ref{thm:hp_bound_MLE_q0}] Most of the work in this proof is to check the conditions of our generic theorem \ref{thm:exp_bound_for_ERM} for empirical risk minimizers, in particular the variance bound (assumption \ref{assumption:var_bound}, the assumption connecting the $\|\cdot\|_\infty$ norm to the norm $\sigma$ (assumption \ref{assumption:sigma_alpha_dominates_sup_norm}), and the entropy bound. In doing so, we follow closely the techniques presented in section 3.4.1, chapter 3.4 of \cite{vdV_Wellner96} for the analysis of maximum likelihood estimators.

\paragraph{Notation.} We introduce the alternative loss
\begin{align}
    \widetilde{\ell}_n(q)(c,o) := - \log \left( \frac{q + q_n}{2 q_n}(o \mid c)\right),
\end{align}
the corresponding empirical and population risks
\begin{align}
    \widehat{\widetilde{R}}_n(q) := \frac{1}{n} \sum_{i=1}^n \widetilde{\ell}_n(q)(\widetilde{C}(k), \widetilde{O}(k)) \qquad \text{and} \qquad \widetilde{R}_{0,n} := \frac{1}{n} \sum_{i=1}^n E_{q_0, \widetilde{h}_k} \left[ \widetilde{\ell}_n(q)(\widetilde{C}(k), \widetilde{O}(k)) \right].
\end{align}
For any $c$ and any two conditional densities $q_1(\cdot \mid c)$ and $q_1(\cdot \mid c)$, we introduce the conditional Hellinger distance:
\begin{align}
    H(q_1,q_2 \mid c) := \left( \int \left( \sqrt{q_1(o,c)} - \sqrt{q_2(o,c)} \right)^2 do \right)^{1/2}.
\end{align}
For any marginal density $h :\mathcal{C} \to \mathbb{R}$, and any two $q_1$ and $q_2$, we define the conditional Hellinger distance integrated against $h$:
\begin{align}
H_h(q_1,q_2) := \left( \int H^2(q_1, q_2 \mid c) h(c) dc \right)^{1/2}.
\end{align}
We further define
\begin{align}
    H_n(q_1, q_2 \mid c) = H(q_1 + q_n, q_2 + q_n \mid c) \qquad \text{and} \qquad H_{h,n} (q_1,q_2) = H_h(q_1 + q_n, q_2 +q_n).
\end{align}
For a conditional density $q(\cdot \mid c)$, a positive number number $p \geq 1$, let, for any $f:\mathcal{O}  \times \mathcal{C} \to \mathbb{R}$,
\begin{align}
    &\left\lVert f(\cdot, c) \right\rVert_{q( \cdot \mid c), p} := \left( \int \left\lvert f(c,o)\right\rvert^p q(o \mid c) do \right)^{1/p}, \\
    \text{and} &\left\lVert f(\cdot \mid c) \right\rVert_{q( \cdot \mid c), B} := \left( \sum_{p \geq 2} \frac{\left\lVert f(\cdot, c) \right\rVert_{q( \cdot \mid c), p}^p}{p!} \right)^{1/2},
\end{align}
be the $L_p$ norm, and the so-called Bernstein ``norm'' \footnote{It is not actually a norm, but this doesn't matter for what follows.} with respect to $q(\cdot \mid c)$. 

\paragraph{Checking the entropy condition.} We have that 
\begin{align}
    & \left\lVert \widetilde{\ell}_n(q_1) - \widetilde{\ell}_n(q_2) (\cdot, c) \right\rVert_{q_0(\cdot \mid c), 2}^2 \\
    \leq & \left\lVert \widetilde{\ell}_n(q_1) - \widetilde{\ell}_n(q_2) (\cdot, c) \right\rVert_{q_0(\cdot \mid c), B}^2 \\
     \lesssim & H_n^2(q_1, q_2 \mid c) \\
    =& \int \left( \frac{q_1 - q_2}{\sqrt{q_1 + q_n} + \sqrt{q_2 + q_n}} \right)^2(o \mid c) do \\
    \lesssim & \int (q_1 - q_2)^2(o \mid c) do.
\end{align}

The inequality in the third line above is proven to hold under assumption \ref{assumption:ratio_q0_qn} in section 3.4.1 of \cite{vdV_Wellner96}. The last line follows assumption \ref{assumption:lower_bound_pop_MLE}.

Let $\mu_{\mathcal{O}}$ be the Lebesgue measure on $\mathcal{O}$. Integrating the previous inequality against $\widetilde{h}_i$, and recalling the definition of $\sigma$, we have that
\begin{align}
    (1 + 2 \bm{\rho})^{-1/2} \sigma\left(\widetilde{\ell}_n(q_1) - \widetilde{\ell}_n(q_2)\right) 
    := & \sup_{i \geq 1} \left\lVert \widetilde{\ell}_n(q_1) - \widetilde{\ell}_n(q_2)  \right\rVert_{q_0, \widetilde{h}_i, 2} \\
    \lesssim &\sup_{i \geq 1} \left\lVert q_1 - q_2 \right\rVert_{\mu_{\mathcal{O}}, \widetilde{h}_i, 2} \\
    \lesssim & \| q_1 - q_2 \|_{\mu, 2},
\end{align}
where the last inequality follows from assumption \ref{assumption:unif_boundedness_h_i}. Therefore, denoting $\widetilde{\ell}_n(\mathcal{Q}) := \{ \widetilde{\ell}_n(q) : q \in \mathcal{Q} \}$, we have that
\begin{align}
    \log N_{[\,]}(\epsilon, \widetilde{\ell}_n(\mathcal{Q}), \sigma) \lesssim &\log N_{[\,]}(\epsilon, \mathcal{Q}, L_2(\mu)) \\
    \lesssim & \epsilon^{-1} \log (1 / \epsilon))^{2(d-1)},
\end{align}
where the last inequality is the claim of proposition \ref{prop:bkting_entropy_HAL}.

\paragraph{Checking the variance bound condition.} The first claim of theorem 3.4.4 in \cite{vdV_Wellner96} asserts that
\begin{align}
    H^2(q,q_n \mid c) \lesssim \int \left( \widetilde{\ell}_n(q) - \widetilde{\ell}_n(q_n)\right)(c,o) q_0(o \mid c) do.
\end{align}
In section 3.4.1 of \cite{vdV_Wellner96}, the authors also show the following claim, which we transpose to our notation:
\begin{align}
    \left\lVert \left( \widetilde{\ell}_n(q) - \widetilde{\ell}_n(q_n)\right)(c, \cdot)  \right\rVert_{q_0(\cdot \mid c), B} \lesssim H^2(q,q_n \mid c).
\end{align}
Therefore, putting the previous two inequalities together, integrating w.r.t. $\bar{\widetilde{h}}_n$ and recalling the definition of $\widetilde{R}_{0,n}$, we have that
\begin{align}
    \left\lVert  \widetilde{\ell}_n(q) - \widetilde{\ell}_n(q_n) \right\rVert_{q_0, \bar{\widetilde{h}}_n, 2} \lesssim \widetilde{R}_{0,n}(q) - \widetilde{R}_{0,n}(q_n).
\end{align}
Using assumption \ref{assumption:bound_h_i_over_hbar_n} then yields that
\begin{align}
    \sigma^2(q, q_n) \lesssim \widetilde{R}_{0,n}(q) - \widetilde{R}_{0,n}(q_n),
\end{align}
which is the wished variance bound condition.

\paragraph{Checking assumption \ref{assumption:sigma_alpha_dominates_sup_norm}.} Lemma \ref{lemma:sup_norm_bound_Lipschitz} gives us that assumption \ref{assumption:sigma_alpha_dominates_sup_norm} holds for $\alpha = 1/(d+1)$.

\paragraph{Upper bounding the rate of convergence.} We calculate $r_n$ defined in theorem \ref{thm:exp_bound_for_ERM}. Observing that, from lemma \ref{prop:bkting_entropy_HAL}, it holds for any $\nu > 0$ that $\log N_{[\,]}(\epsilon, \mathcal{Q}, \sigma) \lesssim \epsilon^{-(1+\nu)}$, we have 
\begin{align}
    \phi_n(r) = \frac{r^{\frac{1}{d+1}}}{\sqrt{n}} + \log n \frac{r^{\frac{1-\nu}{2}}}{\sqrt{n}} + \frac{(\log n)^2}{n} r^{\frac{1}{d+1} - 1 - \nu}.
\end{align}
For $d \geq 2$, $\nu$ small enough, and $n$ large enough, it is straightforward to observe 
\begin{align}
    n^{-\frac{2}{4 - 2 \alpha}} \gtrsim \phi_n( n^{-\frac{1}{4 - 2 \alpha}} ),
\end{align}
which, from lemma \ref{lemma:sublinear_functions} implies that $r_n \lesssim n^{-\frac{1}{4 - 2 \alpha}}$. We have thus checked the assumptions of theorem \ref{thm:exp_bound_for_ERM} and shown that $r_n$ is upper bounded by the wished rate, which implies the claim.
\end{proof}

\section{Proof of the type-I error guarantee for the adaptive stopping rule (theorem \ref{thm:type_I_error_adaptive_stopping})}

\begin{proof}[Proof of theorem \ref{thm:type_I_error_adaptive_stopping}] Under $H_0$, we have that $\Psi(P_0^{T,N}) = 0$. That the procedure rejects is therefore equivalent to the following event:
\begin{align}
    & \left\lbrace \exists T \in [t_0 T_{\max}, T_{\max}],\ \sqrt{\frac{T}{T_{\max}}} \sqrt{T N} \frac{ \left( \widehat{\Psi}_{T,N} - \Psi(P_0^{T,N}) \right) }{\sigma_{0,\infty,N}} \in [\pm a_\alpha(T / T_{\max})] \right\rbrace \\
    =& \left\lbrace \exists T \in [t_0 T_{\max}, T_{\max}],\ \frac{T}{T_{\max}} \sqrt{T_{\max} N} \frac{ \left( \widehat{\Psi}_{T,N} - \Psi(P_0^{T,N}) \right) }{\sigma_{0,\infty,N}} \in [\pm a_\alpha(T / T_{\max})] \right\rbrace \\
    =& \left\lbrace \sup_{t \in [t_0,1]} t \sqrt{T_{\max} N}  \frac{ \left\lvert \widehat{\Psi}_{t T_{\max},N} - \Psi(P_0^{t T_{\max},N}) \right\rvert }{\sigma_{0,\infty,N} a_{\alpha}(t)} \leq 1 \right\rbrace
\end{align}

Let $\phi$ be the function defined on the set $\mathbb{D}([0,1])$ of real-valued cadlag functions on $[0,1]$ by 
\begin{align}
    \phi:f \mapsto \sup_{t \in [t_0, 1]} \frac{|f(t)|}{a_\alpha(t)}.
\end{align}

For $([\pm a_\alpha(t) : t \in [0,1])$ to be an $(1-\alpha)$ joint confidence band for $(W(t) : t \in [0,1])$, each $[\pm a_\alpha(t)]$ must contain $W(t)$ with probability at least $1-\alpha$, and therefore, for every $t \geq t_0$, we must have $a_\alpha(t) \geq \sqrt{t_0} q_{1-\alpha/2}$, where $q_{1-\alpha/2}$ is the $1-\alpha/2$ quantile of the standard normal. Therefore, the denominator in the definition of $\phi$ remains uniformly in $t$ bounded away from 0, thus ensuring that $\phi$ is continous w.r.t $\|\cdot\|_\infty$, and bounded.

Therefore, from the fact that, 
\begin{align}
    \left\lbrace t \sqrt{T_{\max} N} \left( \widehat{\Psi}_{N,T} - \Psi(P_0^{T,N}) \right) : t \in [t_0, 1] \right\rbrace \xrightarrow{d} W,
\end{align}
by definition of weak convergence, and continuity and boundedness of $\phi$ as a mapping $(\mathbb{D}([0,1]), \|\cdot\|_\infty \to (\mathbb{R}, | \cdot |)$, we have that
\begin{align}
     \lim_{T_{\max} \to \infty} P_0 \left[ \sup_{t \in [t_0,1]} t \sqrt{T_{\max} N}  \frac{ \left\lvert \widehat{\Psi}_{t T_{\max},N} - \Psi(P_0^{t T_{\max},N}) \right\rvert }{\sigma_{0,\infty,N} a_{\alpha}(t)} \leq 1 \right] = P_0 \left[ \sup_{t \in [t_0,1]} \frac{|W(t)|}{ a_\alpha(t)} \leq 1 \right] \geq 1-\alpha,
\end{align}
where the latter inequality follows by definition of $a_\alpha(t)$
\end{proof}

$$ P \left[ \forall n \geq 1, \| \bar{D} (\widehat{q}_n ) - \bar{D}(q_0) \|_\infty \leq C(\epsilon) a_n \right] \geq 1 - \epsilon.$$
\end{document}